\documentclass[10pt]{article}
\usepackage[margin=1in]{geometry}
\usepackage{amsmath,amssymb,amsthm,mathtools}
\usepackage{enumitem}
\usepackage{authblk}
\usepackage[hidelinks]{hyperref}
\usepackage{aliascnt}

\numberwithin{equation}{section}

\newtheorem{theorem}{Theorem}[section]

\newaliascnt{lemma}{theorem}
\newtheorem{lemma}[lemma]{Lemma}
\aliascntresetthe{lemma}

\newaliascnt{proposition}{theorem}
\newtheorem{proposition}[proposition]{Proposition}
\aliascntresetthe{proposition}

\newaliascnt{corollary}{theorem}
\newtheorem{corollary}[corollary]{Corollary}
\aliascntresetthe{corollary}

\newaliascnt{conjecture}{theorem}
\newtheorem{conjecture}[conjecture]{Conjecture}
\aliascntresetthe{conjecture}

\newaliascnt{hypothesis}{theorem}

\aliascntresetthe{hypothesis}

\newaliascnt{remark}{theorem}

\aliascntresetthe{remark}

\newaliascnt{definition}{theorem}
\newtheorem{definition}[definition]{Definition}
\aliascntresetthe{definition}

\newaliascnt{claim}{theorem}
\newtheorem{claim}[claim]{Claim}
\aliascntresetthe{claim}

\newtheorem{problem}[theorem]{Problem}

\usepackage[capitalize,nameinlink]{cleveref}
\crefname{theorem}{Theorem}{Theorems}
\Crefname{theorem}{Theorem}{Theorems}
\crefname{lemma}{Lemma}{Lemmas}
\Crefname{lemma}{Lemma}{Lemmas}
\crefname{proposition}{Proposition}{Propositions}
\Crefname{proposition}{Proposition}{Propositions}
\crefname{corollary}{Corollary}{Corollaries}
\Crefname{corollary}{Corollary}{Corollaries}
\crefname{conjecture}{Conjecture}{Conjectures}
\Crefname{conjecture}{Conjecture}{Conjectures}
\crefname{hypothesis}{Hypothesis}{Hypotheses}
\Crefname{hypothesis}{Hypothesis}{Hypotheses}
\crefname{remark}{Remark}{Remarks}
\Crefname{remark}{Remark}{Remarks}
\crefname{section}{Section}{Sections}
\Crefname{section}{Section}{Sections}
\crefname{equation}{equation}{equations}
\Crefname{equation}{Equation}{Equations}
\crefname{definition}{definition}{definitions}
\Crefname{definition}{Definition}{Definitions}
\crefname{claim}{claim}{claims}
\Crefname{claim}{Claim}{Claims}

\newcommand{\abs}[1]{\left|#1\right|}
\newcommand{\norm}[1]{\left\lVert#1\right\rVert}

\title{\Large\bf The spectral radius of $k$-chromatic $r$-graphs}
\author{Xizhi~Liu\thanks{Supported by the Excellent Young Talents Program (Overseas) of the National Natural Science Foundation of China. Email: \texttt{liuxizhi@ustc.edu.cn}} } 
\author{Junchi~Luo\thanks{Email: \texttt{ljc2022@mail.ustc.edu.cn}}}
\affil{\small School of Mathematical Sciences, University of Science and Technology of China, Hefei, China}
\date{\today}
\begin{document}
\maketitle
\begin{abstract}
For an $r$-uniform hypergraph $G$, let $\lambda^{(p)}(G)$ denote its $p$-spectral radius, defined as the maximum of the polyform of $G$ over the unit sphere in the $\ell_p$-norm.
Let $Q_k^r(n)$ be the complete $k$-chromatic $r$-graph on $n$ vertices with color classes as equal as possible.
Kang--Nikiforov--Yuan conjectured that, for every $p\ge1$ and $n>(r-1)k$, the $r$-graph $Q_k^r(n)$ is the unique maximizer of $\lambda^{(p)}$ among all $k$-chromatic $r$-graphs of order $n$.
They also conjectured the corresponding explicit bound
\[
        \lambda^{(p)}(G)
        \le
        r!\left(\tbinom nr-k\tbinom{n/k}{r}\right)n^{-r/p},
\]
with equality only in the divisible extremal case.
The case $r=3$ was established in their work.
This paper resolves the remaining cases $r\ge4$, and hence settles both conjectures for all $r\ge3$.
As a consequence, the same threshold gives an anti-Wilf-type spectral certificate: any $r$-graph of order $n$ whose $p$-spectral radius exceeds the displayed bound has chromatic number at least $k+1$.
\end{abstract}
\section{Introduction}
\label{sec:introduction}

An $r$-uniform hypergraph, or briefly an $r$-graph, is a pair $G=(V,E)$ with $E\subseteq V^{(r)}$, where $V^{(r)}$ denotes the family of all $r$-subsets of $V$.
We identify a hypergraph with its edge set whenever convenient, and write $|G|$ for the number of edges of $G$.
For a vector $x=(x_v)_{v\in V}\in\mathbb R^V$, define the polyform
\[
        P_G(x)
        =
        r!\sum_{e\in G}\prod_{v\in e}x_v .
\]
For a real number $p\ge1$, the $p$-spectral radius of $G$ is
\[
        \lambda^{(p)}(G)
        =
        \max_{\sum_{v\in V}|x_v|^p=1} P_G(x).
\]
Since all coefficients of $P_G$ are nonnegative, replacing $x$ by $(|x_v|)_{v\in V}$ cannot decrease $P_G(x)$.
Thus the maximum in the definition of $\lambda^{(p)}(G)$ may always be taken over nonnegative vectors.

The parameter $\lambda^{(p)}$ forms a useful bridge between spectral and extremal hypergraph theory (see the survey of Nikiforov~\cite{Nikiforov2014}).
When $p=1$, it is the Lagrangian of $G$; when $p=r$, it is the variational spectral radius of $G$; and as $p\to\infty$, it approaches $r!|G|$.
Variational and tensorial approaches to hypergraph eigenvalues were developed in several directions, including the work of Friedman--Wigderson~\cite{FriedmanWigderson1995}, Lim~\cite{Lim2005}, Qi~\cite{Qi2005,Qi2006}, Cooper--Dutle~\cite{CooperDutle2012}, and Pearson--Zhang~\cite{PearsonZhang2014}.
The $p$-spectral radius in the above form was introduced by Keevash--Lenz--Mubayi~\cite{KLM2014} and was developed further by Nikiforov~\cite{Nikiforov2014}.

The study of spectral parameters and graph colorings has a long history.
For ordinary graphs, the Motzkin--Straus theorem identifies the Lagrangian in terms of the clique number~\cite{MotzkinStraus1965}, while Wilf's theorem bounds the chromatic number in terms of the spectral radius~\cite{Wilf1967}.
Further inequalities involving the chromatic number, clique number, and spectral radius were obtained by Cvetkovi{\'c}~\cite{Cvetkovic1972}, Edwards--Elphick~\cite{EdwardsElphick1983}, Nikiforov~\cite{Nikiforov2002, Nikiforov2007}, Feng--Li--Zhang~\cite{FengLiZhang2007}, Wocjan--Elphick~\cite{WocjanElphick2013}, and Elphick--Wocjan~\cite{ElphickWocjan2017}.
For hypergraphs, extremal problems for $t$-partite and $t$-colorable structures were studied by Mubayi--Talbot~\cite{MubayiTalbot2008}, and the $p$-spectral analogue for $k$-partite and $k$-chromatic uniform hypergraphs was investigated by Kang--Nikiforov--Yuan~\cite{KNY2015}.

An $r$-graph is called $k$-chromatic if its vertex set can be partitioned into $k$ classes so that no edge is contained in a single class.
Given nonnegative integers $n_1+\cdots+n_k=n$, let $Q(n_1,\ldots,n_k)$ denote the complete $k$-chromatic $r$-graph with color classes $V_1,\ldots,V_k$, where $|V_i|=n_i$, whose edges are precisely all $r$-sets not contained in a single color class.
Let $Q_k^r(n)$ denote the member for which the color class sizes differ by at most one.
Thus $Q_k^r(n)$ is the balanced complete $k$-chromatic $r$-graph.

Kang--Nikiforov--Yuan~\cite{KNY2015} established the exact $k$-chromatic theorem for $3$-graphs and proposed the following conjecture for general uniformity.

\begin{conjecture}[{\cite[Conjecture~7]{KNY2015}}]
\label{conj:kny}
Let $k\ge2$, $r\ge 4$, and let $G$ be a $k$-chromatic $r$-graph of order $n>(r-1)k$.
For every $p\ge1$,
\[
        \lambda^{(p)}(G)
        <
        \lambda^{(p)}(Q_k^r(n)),
\]
unless $G$ is isomorphic to $Q_k^r(n)$.
\end{conjecture}

They also conjectured the following explicit version.

\begin{conjecture}[{\cite[Conjecture~8]{KNY2015}}]
\label{conj:kny-explicit}
Let $k\ge2$, $r\ge 4$, and let $G$ be a $k$-chromatic $r$-graph of order $n>(r-1)k$.
For every $p\ge1$,
\[
        \lambda^{(p)}(G)
        <
        r!\left(
        \tbinom nr-k\tbinom{n/k}{r}
        \right)n^{-r/p},
\]
unless $k\mid n$ and $G$ is isomorphic to $Q_k^r(n)$.
\end{conjecture}

Note that if $n\le (r-1)k$, then the complete $r$-graph $K_n^r$ is already $k$-chromatic, since its vertices may be partitioned into $k$ classes of size at most $r-1$.
Thus the exact problem in that range reduces to the complete-graph case.

Our first main theorem confirms Conjecture~\ref{conj:kny}.

\begin{theorem}
\label{thm:main}
Let $r\ge4$, $k\ge2$, $p\ge1$, and $n>(r-1)k$.
If $G$ is a $k$-chromatic $r$-graph of order $n$, then
\[
        \lambda^{(p)}(G)
        \le
        \lambda^{(p)}(Q_k^r(n)),
\]
with equality if and only if $G$ is isomorphic to $Q_k^r(n)$.
\end{theorem}

We also prove the following evaluation bound for the conjectured extremal graph, which combined with Theorem~\ref{thm:main}, confirms  Conjecture~\ref{conj:kny-explicit} for all $r\ge4$.

\begin{theorem}
\label{thm:evaluation-Qkr}
Let $r\ge4$, $k\ge2$, $n>(r-1)k$, and $p\ge1$.
Then
\[
        \lambda^{(p)}(Q_k^r(n))
        \le
        r!\left(
        \binom nr-k\binom{n/k}{r}
        \right)n^{-r/p}.
\]
If $k\mid n$, equality holds.
If $k\nmid n$, the inequality is strict.
\end{theorem}

As an immediate consequence, the explicit bound gives an anti-Wilf-type spectral coloring criterion for uniform hypergraphs.

\begin{corollary}
\label{cor:spectral-coloring-criterion}
Let $r\ge4$, $k\ge2$, $p\ge1$, and $n>(r-1)k$.
Let $G$ be an $r$-graph of order $n$.
If
\[
        \lambda^{(p)}(G)
        >
        r!\left(
        \tbinom nr-k\tbinom{n/k}{r}
        \right)n^{-r/p},
\]
then $G$ is not $k$-chromatic. Equivalently, $\chi(G)\ge k+1$. 
\end{corollary}

We briefly describe the proof.
First, by passing to the complete $k$-chromatic graph on the same color classes, it is enough to solve the complete problem; strict monotonicity inside complete $k$-chromatic graphs then gives the equality case.
The proof for $p>1$ uses the differentiability of the $\ell_p$-sphere.
In an extremal complete $k$-chromatic graph, a positive eigenvector is constant on each color class.
After ordering the classes by size, we prove that the corresponding class values are ordered in the opposite direction and that the class $p$-masses satisfy sharp one-sided inequalities.
These structural facts allow us to perform an exact local smoothing operation: if two extreme classes have sizes differing by at least two, then one can move one vertex from the larger class to the smaller class and redistribute the relevant local $p$-masses so that the polyform strictly increases.
This contradicts extremality and forces the color classes to be balanced.

The endpoint $p=1$ is handled separately.
Here the Lagrangian constraint is linear, and the previous $p>1$ smoothing argument is no longer available.
Instead we prove a strict two-class balancing inequality.
If two color classes have sizes $a$ and $b$ with $a\ge b+2$, then moving one vertex from the larger class to the smaller one and redistributing the total $\ell_1$-mass of these two classes uniformly increases the mixed two-class contribution, while all external contributions do not decrease by Schur-concavity.
This again forces the extremal complete $k$-chromatic graph to have balanced color classes.

Finally, Theorem~\ref{thm:evaluation-Qkr} is proved by a separate argument.
The case $p>1$ follows from a H{\"o}lder reduction once the $p=1$ estimate is known.
The Lagrangian estimate for $Q_k^r(n)$ is reduced to a one-variable optimization depending on the total mass placed on the larger color classes.
The maximum is localized to an interval between the equal-class-mass point and the equal-vertex-mass point, and the desired inequality is then proved by a Bernstein-coefficient argument together with a sharp endpoint estimate.

The paper is organized as follows.
Section~\ref{sec:Preliminaries} collects the spectral tools and quoted results used later.
Section~\ref{sec:structural} derives structural consequences for extremal complete $k$-chromatic graphs when $p>1$.
Section~\ref{sec:smoothing} proves Theorem~\ref{thm:main} in the case $p>1$, using the local gap-two smoothing argument.
Section~\ref{sec:pone} proves Theorem~\ref{thm:main} at the Lagrangian endpoint $p=1$.
Section~\ref{sec:evaluation-Qkr} proves Theorem~\ref{thm:evaluation-Qkr}.
Section~\ref{sec:concluding} records several related open directions.

\section{Preliminaries}
\label{sec:Preliminaries}

We collect here the standard facts on $\lambda^{(p)}$ and the known $k$-chromatic results which will be used throughout the proof.
Throughout the paper we use the convention
\[
        \tbinom{x}{s}
        =
        \tfrac{x(x-1)\cdots(x-s+1)}{s!}
\]
when $x$ is real and $s$ is a nonnegative integer, and $\binom ms=0$ for integers $m<s$.


For a vertex $u\in V(G)$, the link of $u$ in $G$ is
\[
        \mathcal L_G(u)
        =
        \{S\subseteq V(G)\setminus\{u\}: |S|=r-1,\ S\cup\{u\}\in G\}.
\]
The following basic facts are standard; see Nikiforov~\cite{Nikiforov2014} and Kang--Nikiforov--Yuan~\cite{KNY2015}.

\begin{lemma}[{\cite[Section~1.3 and Propositions~2.2--2.3]{Nikiforov2014}}]
\label{lem:quoted-basic-variational}
Let $p\ge 1$ be a real number, and let $G$ be an $r$-graph.
\begin{enumerate}[label=(\roman*)]
    \item The maximum defining $\lambda^{(p)}(G)$ is attained, and it is
    attained by a nonnegative vector.
    \item For every real vector $x$, $P_G(x)\le \lambda^{(p)}(G)\|x\|_p^r$. 
    \item If $H\subseteq G$, then $\lambda^{(p)}(H)\le \lambda^{(p)}(G)$. 
\end{enumerate}
\end{lemma}

For $p>1$, eigenvectors satisfy the usual Lagrange multiplier equations.

\begin{lemma}[{\cite[Theorem~3.1]{Nikiforov2014}}]
\label{lem:quoted-eigenequations}
Let $p>1$ be a real number, and let $G$ be an $r$-graph.
If $x=(x_v)_{v\in V(G)}$ is an eigenvector to $\lambda^{(p)}(G)$, normalized by $\sum_v |x_v|^p=1$, then for every $u\in V(G)$,
\[
        \lambda^{(p)}(G)x_u |x_u|^{p-2}
        =
        \frac1r\,\frac{\partial P_G(x)}{\partial x_u}.
\]
In particular, when $x\ge 0$,
\[
        \lambda^{(p)}(G)x_u^{p-1}
        =
        (r-1)!
        \sum_{S\in\mathcal L_G(u)}\prod_{v\in S}x_v .
\]
\end{lemma}

We shall also use the standard symmetry principle for equivalent vertices.

\begin{lemma}[{\cite[Lemma~10]{KNY2015}}]
\label{lem:quoted-symmetric-vertices}
Let $p>1$ be a real number, and let $G$ be an $r$-graph with at least one edge.
Suppose that the transposition of two vertices $u,v\in V(G)$ is an automorphism of $G$.
If $x\ge 0$ is an eigenvector to $\lambda^{(p)}(G)$, then $x_u=x_v$. 
\end{lemma}


We next quote the two structural facts about complete $k$-chromatic $r$-graphs that are used repeatedly in the reduction from arbitrary $k$-chromatic graphs to complete ones.

\begin{lemma}[{\cite[Propositions~13--14]{KNY2015}}]
\label{lem:quoted-complete-chromatic-positivity}
Let $r\ge 3$, $k\ge 2$, and let $Q$ be a complete $k$-chromatic $r$-graph with at least one edge.
Let $p\ge 1$.
Every nonnegative eigenvector to $\lambda^{(p)}(Q)$ is positive.
Moreover, it is constant on each nonempty colour class.
\end{lemma}

\begin{lemma}[{\cite[Propositions~12 and~14]{KNY2015}}]
\label{lem:quoted-strict-monotonicity}
Let $r\ge 3$, $k\ge 2$, $p\ge 1$, and let $Q$ be a complete $k$-chromatic $r$-graph with at least one edge.
If $H$ is a proper subgraph of $Q$ on the same vertex set, then $\lambda^{(p)}(H)<\lambda^{(p)}(Q)$. 
\end{lemma}







We shall use the Power Mean inequality, the AM--GM inequality, Bernoulli's inequality, Maclaurin's inequality, H{\"o}lder's inequality, and the standard majorization facts for Schur-convex and Schur-concave functions.
We refer to Hardy--Littlewood--P{\'o}lya~\cite{HLP} for these classical inequalities.

In particular, we shall use the following direct consequence of H{\"o}lder's inequality.

\begin{lemma}
\label{lem:quoted-holder-reduction}
Let $H$ be an $r$-graph and let $p>1$.
Then
\[
        \lambda^{(p)}(H)
        \le
        \bigl(r!|H|\bigr)^{1-1/p}
        \bigl(\lambda^{(1)}(H)\bigr)^{1/p}.
\]
\end{lemma}

\begin{proof}
Let $x\ge 0$ be an eigenvector to $\lambda^{(p)}(H)$, normalized by $\sum_v x_v^p=1$.
By H{\"o}lder's inequality,
\[
\lambda^{(p)}(H)
=
r!\sum_{e\in H}\prod_{v\in e}x_v
\le
r!|H|^{1-1/p}
\left(\sum_{e\in H}\prod_{v\in e}x_v^p\right)^{1/p}.
\]
Setting $y_v=x_v^p$, we have $y_v\ge 0$ and $\sum_v y_v=1$, hence
\[
        r!\sum_{e\in H}\prod_{v\in e}y_v
        \le \lambda^{(1)}(H).
\]
Combining the last two displays gives the result.
\end{proof}

\begin{lemma}
\label{lem:no-small-class-unified}
Let $r\ge3$, $k\ge2$, $p\ge1$, and $n>(r-1)k$.
Let $\mathbf n=(n_1,\ldots,n_k)$ maximize
\[
        \lambda^{(p)}(Q(n_1,\ldots,n_k))
\]
among all $k$-tuples of nonnegative integers summing to $n$.
Then $n_i\ge r-1$ for every $i \in [k]$. 
\end{lemma}

\begin{proof}
Because $n>(r-1)k$, some complete $k$-chromatic $r$-graph of order $n$ has an edge.
Thus an extremal tuple has positive $p$-spectral radius, and its graph has at least one edge.
By \cref{lem:quoted-complete-chromatic-positivity}, an eigenvector is positive and constant on each nonempty class.

Suppose that some class $V_j$ has size $b\le r-2$.
Since $\sum_i n_i>(r-1)k$, some class $V_i$ has size $a\ge r$.
Let $\alpha>0$ be the common eigenvector value on $V_i$.
Move one vertex $v$ from $V_i$ to $V_j$, and keep the same vector on the same labelled vertex set.
The new graph gains all sets
\[
        \{v\}\cup S,\qquad S\in (V_i\setminus\{v\})^{(r-1)},
\]
and loses all sets
\[
        \{v\}\cup T,\qquad T\in V_j^{(r-1)}.
\]
Since $b\le r-2$, no latter sets exist.
Hence the old vector gives a strictly larger test value for the new tuple, contradicting extremality.
Therefore every class has size at least $r-1$.
\end{proof}

\section{Structural consequences of extremality for \texorpdfstring{$p>1$}{p>1}}
\label{sec:structural}

Throughout this section and the next two sections, we fix $p>1$.
For a $k$-tuple $\mathbf n=(n_1,\ldots,n_k)$ of nonnegative integers summing to $n$, put
\[
        \Lambda_p(\mathbf n)
        =
        \lambda^{(p)}(Q(n_1,\ldots,n_k)).
\]
We call $\mathbf n$ extremal if it maximizes $\Lambda_p$ among all nonnegative integer $k$-tuples with sum $n$.

The main result of this section is the following result. 

\begin{proposition}
\label{prop:structural-hypotheses}
Let $r\ge 3$, $k\ge2$, $p>1$, and $n>(r-1)k$.
Let $\mathbf n=(n_1,\ldots,n_k)$ be an extremal tuple for $\Lambda_p$, and let $x\ge0$ be an eigenvector to $\lambda^{(p)}(Q(\mathbf n))$ normalized by $\sum_v x_v^p=1$.
Then every class has positive size, $x$ is positive, and $x$ is constant on every class.
If $a_i$ denotes the common value of $x$ on $V_i$, then after relabelling the classes so that $n_1\ge\cdots\ge n_k$, the following hold:
\begin{enumerate}[label=(S\arabic*)]
\item $0<a_1\le a_2\le\cdots\le a_k$;
\item for every $i\in[k]$,
\[
        n_1a_1^p\ge n_i a_i^p\ge (n_1-1)a_1^p;
\]
\item
\[
        n_k\ge n_1-\left\lceil\frac1{p-1}\right\rceil.
\]
\end{enumerate}
\end{proposition}

Note that, since $\left\lceil\frac1{p-1}\right\rceil \le 1$ when $p \ge 2$, Proposition~\ref{prop:structural-hypotheses} confirms Conjecture~\ref{conj:kny} for $p \ge 2$.

The proof of Proposition~\ref{prop:structural-hypotheses} is divided into several lemmas. 

The next result shows that, for an extremal complete $k$-chromatic graph, the class values of a positive eigenvector are ordered oppositely to the class sizes.
For a finite multiset $W=\{w_1,\ldots,w_m\}$, let $e_q(W)$ denote the $q$-th elementary symmetric sum, with $e_0(W)=1$ and $e_q(W)=0$ for $q>|W|$.
Explicitly,
\[
        e_q(W)
        =
        \sum_{I \in \binom{[m]}{q}}\prod_{i\in I}w_i .
\]
We shall use the identity
\[
        e_s(W\cup\{z\})=e_s(W)+z e_{s-1}(W).
\]

\begin{lemma}\label{lem:opposite-order}
Let $p>1$ be a real number, let $\mathbf n=(n_1,\ldots,n_k)$ be an extremal tuple, and let $a_i$ be the common positive eigenvector value on $V_i$.
If $n_i\ge n_j$, then $a_i\le a_j$.
In particular, if $n_i=n_j$, then $a_i=a_j$.
\end{lemma}

\begin{proof}
Put $s=r-1$.
For a vertex $u\in V_i$, write
\[
 L_i=\sum_{S\in\mathcal L_Q(u)}\prod_{w\in S}x_w.
\]
The Lagrange multiplier equation for the positive eigenvector gives
\[
 \lambda^{(p)}(Q)a_i^{p-1}=(r-1)!L_i.
\]
The analogous equation holds for each class.

Assume, for a contradiction, that $n_i\ge n_j$ but $a_i>a_j$.
Choose $u\in V_i$ and $v\in V_j$.
Let $\mathcal C$ be the multiset of all weights outside $\{u,v\}$.
Then
\[
 L_i=e_s(\mathcal C\cup\{a_j\})-\binom{n_i-1}{s}a_i^s,
\]
and
\[
 L_j=e_s(\mathcal C\cup\{a_i\})-\binom{n_j-1}{s}a_j^s.
\]
Therefore
\[
 L_i-L_j=(a_j-a_i)e_{s-1}(\mathcal C)-\binom{n_i-1}{s}a_i^s+\binom{n_j-1}{s}a_j^s.
\]
The factor $e_{s-1}(\mathcal C)$ is positive: if $s=1$, then it is $e_0(\mathcal C)=1$, while if $s\ge2$, \cref{lem:no-small-class-unified} ensures that after deleting $u$ and $v$ there remain at least $s-1$ positive entries.
Hence the first term is negative, since $a_i>a_j$.
The remaining two terms are nonpositive, because $n_i\ge n_j$ and $a_i>a_j$.
Hence $L_i<L_j$.

On the other hand, the eigenequations give
\[
 L_i-L_j=\frac{\lambda^{(p)}(Q)}{(r-1)!}\left(a_i^{p-1}-a_j^{p-1}\right)>0,
\]
again because $a_i>a_j$.
This contradiction proves the result.
\end{proof}

We next prove the main mass-ordering lemma.
The following coefficient inequality is the elementary input.

\begin{lemma}\label{lem:coeff-monotone}
Let $1\le t$, let $\ell\ge m\ge0$ be integers, and define, for integers $z\ge \ell$,
\[
 F_{\ell,m,t}(z)=\frac{\binom z\ell}{\binom z m}z^{-(\ell-m)/t}.
\]
Then $F_{\ell,m,t}(z)$ is nondecreasing in $z$.
\end{lemma}

\begin{proof}
The case $\ell=m$ is trivial.
Assume $\ell>m$.
For $z>\ell$,
\[
 \frac{F_{\ell,m,t}(z)}{F_{\ell,m,t}(z-1)}=\frac{z-m}{z-\ell}\left(\frac{z-1}{z}\right)^{(\ell-m)/t}.
\]
Also
\[
 \frac{z-m}{z-\ell}=\prod_{q=m+1}^{\ell}\frac{z-q+1}{z-q}\ge\left(\frac{z}{z-1}\right)^{\ell-m}.
\]
Since $t\ge1$, this implies
\[
 \frac{F_{\ell,m,t}(z)}{F_{\ell,m,t}(z-1)}\ge\left(\frac{z}{z-1}\right)^{(\ell-m)(1-1/t)}\ge1.
\]
The endpoint $z=\ell$ is harmless, because the preceding inequality proves monotonicity from that point onward.
\end{proof}

\begin{lemma}\label{lem:t-mass-switch}
Let $Q=Q(n_1,\ldots,n_k)$ be a complete $k$-chromatic $r$-graph, and fix two nonempty classes $U,V$ of sizes $N\ge M\ge1$.
Suppose a nonnegative vector is constant with values $\alpha$ on $U$ and $\beta$ on $V$, while all other coordinates are fixed.
Let $1\le t$ and set
\[
 Y_U=N\alpha^t,
 \qquad
 Y_V=M\beta^t.
\]
If $N>M$ and $Y_U<Y_V$, define new values $\widetilde\alpha,\widetilde\beta$ by
\[
 N\widetilde\alpha^t=Y_V,
 \qquad
 M\widetilde\beta^t=Y_U.
\]
Then the value of $P_Q$ does not decrease after replacing $\alpha,\beta$ by $\widetilde\alpha,\widetilde\beta$.
\end{lemma}

\begin{proof}
Only terms using vertices from $U\cup V$ can change.
Let $Z$ be the multiset of all fixed weights outside $U\cup V$.
For fixed nonnegative integers $\ell,m,h$ with $\ell+m+h=r$ and $\ell+m>0$, the contribution of terms using $\ell$ vertices from $U$, $m$ vertices from $V$, and $h$ vertices from outside $U\cup V$ is a nonnegative common outside factor $e_h(Z)$ times
\[
 \binom N\ell \binom M m \alpha^\ell \beta^m,
\]
except that when $h=0$ the two monochromatic cases $(\ell,m)=(r,0)$ and $(0,r)$ are nonedges both before and after the switch.
Their contribution is therefore identically zero, and we omit this pair from all summations below.
We use the usual convention that $\binom ab=0$ when $b>a$.

It is enough to compare, for every unordered pair $\{\ell,m\}$ with $\ell\ge m$, the sum of the two corresponding contributions.
Write
\[
 C_{\ell,m}=\binom N\ell\binom M m N^{-\ell/t}M^{-m/t}.
\]
In terms of $Y_U,Y_V$, the pair of exponents $(\ell,m)$ contributes $C_{\ell,m}Y_U^{\ell/t}Y_V^{m/t}$, while the pair $(m,\ell)$ contributes $C_{m,\ell}Y_U^{m/t}Y_V^{\ell/t}$.

By \cref{lem:coeff-monotone}, and also trivially when $M<\ell$, we have
\[
 C_{\ell,m}\ge C_{m,\ell}\qquad (N\ge M,\ \ell\ge m).
\]
Since $Y_V>Y_U$ and $\ell\ge m$,
\[
 Y_V^{\ell/t}Y_U^{m/t}-Y_V^{m/t}Y_U^{\ell/t}\ge0.
\]
After the switch $Y_U$ and $Y_V$ exchange their positions, and the change in the paired contribution is
\begin{align*}
&\left(C_{\ell,m}Y_V^{\ell/t}Y_U^{m/t}+C_{m,\ell}Y_V^{m/t}Y_U^{\ell/t}\right)
-\left(C_{\ell,m}Y_U^{\ell/t}Y_V^{m/t}+C_{m,\ell}Y_U^{m/t}Y_V^{\ell/t}\right)\\
&\qquad=(C_{\ell,m}-C_{m,\ell})\left(Y_V^{\ell/t}Y_U^{m/t}-Y_V^{m/t}Y_U^{\ell/t}\right)\ge0.
\end{align*}
Multiplying by the nonnegative outside factor $e_h(Z)$ and summing over all choices of $h,\ell,m$ proves that $P_Q$ does not decrease.
\end{proof}

\begin{lemma}\label{lem:p-mass-order}
Let $p>1$ be a real number, let $\mathbf n=(n_1,\ldots,n_k)$ be an extremal tuple, and let $a_i$ be the common positive eigenvector value on $V_i$.
If $n_i\ge n_j$, then
\[
 n_i a_i^p\ge n_j a_j^p.
\]
\end{lemma}

\begin{proof}
If $n_i=n_j$, then $a_i=a_j$ by \cref{lem:opposite-order}, and the claim is immediate.
Assume $N:=n_i>M:=n_j$.
Let $\alpha=a_i$ and $\beta=a_j$.
We first prove that, for every $1\le t<p$,
\begin{equation}\label{eq:t-mass-order}
 N\alpha^t\ge M\beta^t.
\end{equation}
Suppose not.
Put $Y_i=N\alpha^t$ and $Y_j=M\beta^t$, so that $Y_i<Y_j$.
Switch the two $t$-masses as in \cref{lem:t-mass-switch}.
Thus the new class values $\widetilde\alpha,\widetilde\beta$ satisfy
\[
 N\widetilde\alpha^t=Y_j,
 \qquad
 M\widetilde\beta^t=Y_i.
\]
By \cref{lem:t-mass-switch}, the value of $P_Q$ does not decrease.

We now compare the $p$-norm.
Set $q=p/t>1$.
The old $p$-norm contribution of these two classes is
\[
 N\alpha^p+M\beta^p=N^{1-q}Y_i^q+M^{1-q}Y_j^q,
\]
whereas the new contribution is
\[
 N\widetilde\alpha^p+M\widetilde\beta^p=N^{1-q}Y_j^q+M^{1-q}Y_i^q.
\]
The difference old minus new is
\[
 (Y_j^q-Y_i^q)(M^{1-q}-N^{1-q})>0,
\]
because $Y_j>Y_i$, $N>M$, and $1-q<0$.
Thus the switched vector has $p$-norm strictly smaller than $1$, while its polynomial value is at least $P_Q(x)>0$.
Scaling it back to $p$-norm $1$ strictly increases the value of $P_Q$, contradicting the maximality of $x$ for $\lambda^{(p)}(Q)$.
Therefore \eqref{eq:t-mass-order} holds for every $1\le t<p$.

Letting $t\uparrow p$ in \eqref{eq:t-mass-order} gives $N\alpha^p\ge M\beta^p$, as required.
\end{proof}

\begin{lemma}\label{lem:one-vertex-transfer-pgreater}
Let $p>1$ be a real number, let $\mathbf n=(n_1,\ldots,n_k)$ be an extremal tuple, and order the classes so that $n_1\ge\cdots\ge n_k$.
Let $a_i$ be the common positive eigenvector value on $V_i$.
If $n_i\le n_1-1$, then
\[
 \binom{n_1-1}{r-1}a_1^{r-1}\le \binom{n_i}{r-1}a_i^{r-1}.
\]
Consequently,
\[
 \frac{a_i}{a_1}\ge \frac{n_1-1}{n_i}.
\]
\end{lemma}

\begin{proof}
Move one vertex $v$ from the largest class $V_1$ to $V_i$ and keep the old eigenvector on the same labelled vertices.
Since the original tuple is extremal, the new graph cannot give a larger test value.
The only changed edges are the following.
The new graph gains all sets
\[
 \{v\}\cup S,\qquad S\in (V_1\setminus\{v\})^{(r-1)},
\]
and loses all sets
\[
 \{v\}\cup T,
 \qquad
 T\in V_i^{(r-1)}.
\]
Therefore
\[
 \binom{n_1-1}{r-1}a_1^r\le \binom{n_i}{r-1}a_1a_i^{r-1}.
\]
Dividing by $a_1>0$ gives the first assertion.

Now put $s=r-1$, $N=n_1-1$, and $M=n_i$.
By \cref{lem:no-small-class-unified}, $M\ge s$, and because $n_i\le n_1-1$ we have $N\ge M$.
Hence
\[
 \frac{\binom Ns}{\binom Ms}=\prod_{q=0}^{s-1}\frac{N-q}{M-q}\ge \left(\frac NM\right)^s,
\]
since $(N-q)/(M-q)\ge N/M$ for every $0\le q\le s-1$.
The first assertion therefore implies
\[
 \left(\frac{a_i}{a_1}\right)^s\ge \frac{\binom Ns}{\binom Ms}\ge\left(\frac NM\right)^s.
\]
Taking $s$th roots gives
\[
 \frac{a_i}{a_1}\ge \frac NM=\frac{n_1-1}{n_i}.
\]
\end{proof}

We are now ready to present the proof of Proposition~\ref{prop:structural-hypotheses}. 

\begin{proof}[Proof of Proposition~\ref{prop:structural-hypotheses}]
Since $n>(r-1)k$, the extremal graph $Q(\mathbf n)$ has at least one edge.
Hence \cref{lem:quoted-complete-chromatic-positivity} applies: every nonnegative eigenvector is positive and is constant on each nonempty class.
\Cref{lem:no-small-class-unified} further gives $n_i\ge r-1$ for every $i$, so every class has positive size.

Relabel the classes so that $n_1\ge n_2\ge\cdots\ge n_k$.
Then \cref{lem:opposite-order} gives $0<a_1\le a_2\le\cdots\le a_k$, proving (S1).
For every $i$, \cref{lem:p-mass-order} gives
\[
 n_1a_1^p\ge n_i a_i^p.
\]
If $n_i=n_1$, then $a_i=a_1$ by \cref{lem:opposite-order}, and $n_i a_i^p=n_1a_1^p\ge (n_1-1)a_1^p$.
If $n_i\le n_1-1$, then \cref{lem:one-vertex-transfer-pgreater} gives
\[
 a_i\ge \frac{n_1-1}{n_i}a_1.
\]
Thus
\[
 n_i a_i^p\ge n_i\left(\frac{n_1-1}{n_i}\right)^p a_1^p
=(n_1-1)\left(\frac{n_1-1}{n_i}\right)^{p-1}a_1^p\ge (n_1-1)a_1^p.
\]
This proves (S2).

It remains to prove the diameter bound.
Let $d=n_1-n_k$.
If $d=0$, there is nothing to prove.
Assume $d\ge1$.
By \cref{lem:one-vertex-transfer-pgreater}, applied with $i=k$,
\[
 \frac{a_k}{a_1}\ge \frac{n_1-1}{n_k}.
\]
On the other hand, by \cref{lem:p-mass-order},
\[
 n_k a_k^p\le n_1a_1^p,
\]
and hence
\[
 \frac{a_k}{a_1}\le \left(\frac{n_1}{n_k}\right)^{1/p}.
\]
Combining these two estimates gives
\[
 \frac{n_1-1}{n_k}\le \left(\frac{n_1}{n_k}\right)^{1/p}.
\]
Writing $n_1=n_k+d$, this becomes
\[
 1+\frac{d-1}{n_k}\le \left(1+\frac d{n_k}\right)^{1/p}.
\]
Since $p>1$, the function $z\mapsto z^{1/p}$ is strictly concave on $(0,\infty)$, and therefore
\[
 \left(1+\frac d{n_k}\right)^{1/p}<1+\frac d{p n_k}.
\]
Thus $d-1<d/p$, or equivalently
\[
 d<\frac p{p-1}=1+\frac1{p-1}.
\]
Since $d$ is an integer,
\[
 d\le \left\lceil\frac1{p-1}\right\rceil,
\]
which proves (S3).
\end{proof}

\section{Proof of Theorem~\ref{thm:main} for $p > 1$}
\label{sec:smoothing}

We now prove the local smoothing result that will force an extremal complete $k$-chromatic graph to be balanced.
The local argument uses only the structural properties from \cref{prop:structural-hypotheses}.

Let $m\ge1$.
Consider two local blocks $C$ and $B$ of sizes
\[
 \abs C=m+2,
 \qquad
 \abs B=m.
\]
Let their total $p$-masses be $S_C>0$ and $S_B>0$.
We assume
\begin{equation}\label{eq:rho-assumption}
 \rho:=\frac{S_B}{S_C}\ge \frac{m+1}{m+2}.
\end{equation}
The gap-two smoothing replaces the sizes $(m+2,m)$ by $(m+1,m+1)$ while preserving the two block masses.
Thus the old point $p$-masses are
\[
 \frac{S_C}{m+2}\quad\text{on }C,
 \qquad
 \frac{S_B}{m}\quad\text{on }B,
\]
and the new point $p$-masses are
\[
 \frac{S_C}{m+1}\quad\text{on the new }C\text{-block},
 \qquad
 \frac{S_B}{m+1}\quad\text{on the new }B\text{-block}.
\]

\subsection{Full local layers}

We use the standard majorization notation $u\succ v$.

\begin{lemma}\label{lem:gap-two-majorization}
Under \eqref{eq:rho-assumption}, the old local $p$-mass vector majorizes the new local $p$-mass vector.
\end{lemma}

\begin{proof}
By homogeneity set $S_C=1$ and $S_B=\rho$.
Put
\[
 x=\frac1{m+2},
 \qquad
 y=\frac\rho m,
 \qquad
 x^\#=\frac1{m+1},
 \qquad
 y^\#=\frac\rho{m+1}.
\]
The old vector has $m$ copies of $y$ and $m+2$ copies of $x$.
The new vector has $m+1$ copies of $x^\#$ and $m+1$ copies of $y^\#$.
Since $\rho\ge (m+1)/(m+2)$, we have $y\ge x$.

First assume $\rho\ge1$.
Then $y^\#\ge x^\#$.
For $1\le j\le m$, the sum of the largest $j$ old coordinates is $jy$, whereas the sum of the largest $j$ new coordinates is $jy^\#\le jy$.
If $j=m+t$ with $1\le t\le m+2$, the old partial sum is
\[
 \rho+\frac{t}{m+2},
\]
while the new partial sum is at most
\[
 \rho+\frac{t-1}{m+1}.
\]
The difference is
\[
 \frac{t}{m+2}-\frac{t-1}{m+1}=\frac{m+2-t}{(m+1)(m+2)}\ge0.
\]

Now assume $(m+1)/(m+2)\le \rho\le1$.
Then the new decreasing rearrangement begins with $m+1$ copies of $x^\#$.
If $1\le j\le m$, then $jy\ge jx^\#$ because $\rho/m\ge1/(m+1)$.
If $j=m+1$, the old partial sum is
\[
 \rho+\frac1{m+2}\ge1,
\]
which equals the new partial sum.
Finally, if $j=m+1+t$ with $0\le t\le m+1$, then the old minus new partial sum is
\[
 \rho+\frac{t+1}{m+2}-1-\frac{t\rho}{m+1}.
\]
This expression is increasing in $\rho$, and at $\rho=(m+1)/(m+2)$ it equals $0$.
Hence it is nonnegative throughout the interval.
Thus every partial sum of the old decreasing rearrangement is at least the corresponding partial sum of the new decreasing rearrangement.
The total sums are equal, so the old vector majorizes the new one.
\end{proof}

\begin{lemma}
\label{lem:standard-schur-elementary}
Let $0<\alpha<1$.
Then
\[
        z\mapsto e_s(z_1^\alpha,\ldots,z_N^\alpha)
\]
is Schur-concave on $\mathbb R_+^N$.
Moreover, it is strictly Schur-concave on the positive orthant whenever $1\le s\le N$.
\end{lemma}

\begin{proof}
Write
\[
        F_s(z_1,\ldots,z_N)=e_s(z_1^\alpha,\ldots,z_N^\alpha).
\]
The case $s=0$ is trivial, since $F_0\equiv1$.
We first prove Schur-concavity on the positive orthant.
The function $F_s$ is symmetric and continuously differentiable on $(0,\infty)^N$.
By the Schur derivative criterion, it is enough to prove that, for every $i\ne j$,
\[
        (z_i-z_j)\left(
        \frac{\partial F_s}{\partial z_i}
        -
        \frac{\partial F_s}{\partial z_j}
        \right)\le 0.
\]
Fix $i\ne j$.
Let $W$ denote the multiset
\[
        W=\{z_\ell^\alpha:\ell\in[N]\setminus\{i,j\}\}.
\]
We use the conventions $e_0(W)=1$ and $e_q(W)=0$ if $q<0$ or $q>|W|$.
For $s\ge1$, differentiating the elementary symmetric sum gives
\[
\frac{\partial F_s}{\partial z_i}
=
\alpha z_i^{\alpha-1}
\left(e_{s-1}(W)+z_j^\alpha e_{s-2}(W)\right),
\]
and similarly
\[
\frac{\partial F_s}{\partial z_j}
=
\alpha z_j^{\alpha-1}
\left(e_{s-1}(W)+z_i^\alpha e_{s-2}(W)\right).
\]
Hence
\[
\begin{aligned}
\frac{\partial F_s}{\partial z_i}
-
\frac{\partial F_s}{\partial z_j}
&=
\alpha\bigl(z_i^{\alpha-1}-z_j^{\alpha-1}\bigr)e_{s-1}(W)  \\
&\quad
+\alpha\bigl(z_i^{\alpha-1}z_j^\alpha
      -z_j^{\alpha-1}z_i^\alpha\bigr)e_{s-2}(W).
\end{aligned}
\]
Assume, without loss of generality, that $z_i>z_j$.
Since $\alpha-1<0$, we have
\[
        z_i^{\alpha-1}-z_j^{\alpha-1}<0.
\]
Moreover,
\[
        z_i^{\alpha-1}z_j^\alpha
        -z_j^{\alpha-1}z_i^\alpha
        =
        z_i^{\alpha-1}z_j^{\alpha-1}(z_j-z_i)<0.
\]
Since the elementary symmetric sums of the positive multiset $W$ are nonnegative, it follows that
\[
        \frac{\partial F_s}{\partial z_i}
        -
        \frac{\partial F_s}{\partial z_j}
        \le 0.
\]
Thus
\[
        (z_i-z_j)\left(
        \frac{\partial F_s}{\partial z_i}
        -
        \frac{\partial F_s}{\partial z_j}
        \right)\le0.
\]
The Schur derivative criterion therefore proves that $F_s$ is Schur-concave on $(0,\infty)^N$.

The extension to $\mathbb R_+^N$ follows by continuity.
Indeed, if $u,v\in\mathbb R_+^N$ and $u\succ v$, then for every $\varepsilon>0$ we also have
\[
        u+\varepsilon\mathbf 1\succ v+\varepsilon\mathbf 1,
\]
and both vectors lie in $(0,\infty)^N$.
Applying the already proved positive-orthant case and letting $\varepsilon\downarrow0$ gives
\[
        F_s(u)\le F_s(v).
\]
Thus $F_s$ is Schur-concave on $\mathbb R_+^N$.

It remains to prove the strict statement.
Let $1\le s\le N$, and let all coordinates be positive.
We claim that every nontrivial Robin-Hood transfer strictly increases $F_s$.
Suppose $z_i>z_j$, and replace $(z_i,z_j)$ by
\[
        (z_i-\delta,\ z_j+\delta),
        \qquad
        0<\delta\le \frac{z_i-z_j}{2},
\]
leaving all other coordinates fixed.
Along the path
\[
        \phi(t)
        =
        F_s(z_1,\ldots,z_i-t,\ldots,z_j+t,\ldots,z_N),
        \qquad
        0\le t\le\delta,
\]
we have, for $0\le t<\delta$,
\[
        z_i-t>z_j+t.
\]
The derivative computation above gives
\[
        \phi'(t)
        =
        \frac{\partial F_s}{\partial z_j}
        -
        \frac{\partial F_s}{\partial z_i}
        >0.
\]
Here the inequality is strict because $1\le s\le N$ and the remaining coordinates are positive: if $s=1$, the strictness follows from the strict concavity of $x^\alpha$; if $2\le s\le N$, then at least one of $e_{s-1}(W)$ and $e_{s-2}(W)$ is positive.
Consequently,
\[
        \phi(\delta)>\phi(0).
\]
Thus every nontrivial Robin-Hood transfer strictly increases $F_s$.

Finally, if $u,v\in(0,\infty)^N$, $u\succ v$, and $v$ is not a permutation of $u$, then by the standard Robin-Hood characterization of majorization, $v$ can be obtained from $u$ by a finite sequence of nontrivial Robin-Hood transfers, with all intermediate vectors remaining in the positive orthant.
Since $F_s$ strictly increases at each step, we get
\[
        F_s(u)<F_s(v).
\]
This proves the strict Schur-concavity on the positive orthant.
\end{proof}

\begin{corollary}\label{cor:full-layers}
Let $p>1$ be a real number, put $\alpha=1/p$, and let $m\ge1$.
Let $A,B>0$ satisfy
\[
\frac BA\ge \frac{m+1}{m+2}.
\]
Set
\[
u=
\left(
\underbrace{\frac{A}{m+2},\ldots,\frac{A}{m+2}}_{m+2},
\underbrace{\frac{B}{m},\ldots,\frac{B}{m}}_{m}
\right)
\]
and
\[
v=
\left(
\underbrace{\frac{A}{m+1},\ldots,\frac{A}{m+1}}_{m+1},
\underbrace{\frac{B}{m+1},\ldots,\frac{B}{m+1}}_{m+1}
\right).
\]
Then, for every $1\le s\le 2m+2$,
\[
e_s(u_1^\alpha,\ldots,u_{2m+2}^\alpha)
\le
e_s(v_1^\alpha,\ldots,v_{2m+2}^\alpha).
\]
Equivalently,
\[
\begin{aligned}
&\sum_{\ell=0}^{s}
\binom{m+2}{\ell}\binom{m}{s-\ell}
\left(\frac{A}{m+2}\right)^{\alpha\ell}
\left(\frac{B}{m}\right)^{\alpha(s-\ell)}
\\
&\qquad\le
\sum_{\ell=0}^{s}
\binom{m+1}{\ell}\binom{m+1}{s-\ell}
\left(\frac{A}{m+1}\right)^{\alpha\ell}
\left(\frac{B}{m+1}\right)^{\alpha(s-\ell)}.
\end{aligned}
\]
\end{corollary}

\begin{proof}
By \cref{lem:gap-two-majorization}, $u\succ v$.
By \cref{lem:standard-schur-elementary}, the function
\[
z\mapsto e_s(z_1^\alpha,\ldots,z_{2m+2}^\alpha)
\]
is Schur-concave on $\mathbb R_{\ge0}^{2m+2}$.
Hence
\[
e_s(u_1^\alpha,\ldots,u_{2m+2}^\alpha)
\le
e_s(v_1^\alpha,\ldots,v_{2m+2}^\alpha).
\]
Expanding both elementary symmetric sums gives the displayed binomial formula.
\end{proof}



\subsection{One-sided local layers}

\begin{definition}
\label{def:local-layers}
Let $C,B$ be disjoint weighted vertex sets.
Define
\[
F_s(C,B)
=\sum_{\substack{T\subseteq C\cup B\\ |T|=s}}\prod_{v\in T}w_v,
\]
\[
O_s(C,B)=
\sum_{\substack{T\subseteq C\cup B\\ |T|=s\\ T\not\subseteq C}}
\prod_{v\in T}w_v,
\]
and
\[
X_s(C,B)=
\sum_{\substack{T\subseteq C\cup B\\ |T|=s\\
T\cap C\ne\emptyset,\ T\cap B\ne\emptyset}}
\prod_{v\in T}w_v.
\]
\end{definition}

In the gap-two smoothing step, the old one-sided layer has
\[
|C|=m+2,\qquad |B|=m,
\]
while the smoothed one has
\[
|C^\#|=|B^\#|=m+1.
\]
Thus \cref{lem:one-sided} compares $O_s(C,B)$ with $O_s(C^\#,B^\#)$, where the forbidden side is $C$ before smoothing and $C^\#$ after smoothing.

\begin{lemma}\label{lem:one-sided}
Suppose all $s$-subsets of the two local blocks are allowed except those contained entirely in the old $C$-block.
Then the corresponding one-sided local layer of rank $s$ does not decrease under the gap-two smoothing.
If $1\le s\le m+2$, then the one-sided layer increases strictly.
\end{lemma}

\begin{proof}
If $s>2m+2$, then both the old and the new local layers are zero, so there is nothing to prove.
Hence we may assume $0\le s\le2m+2$.

For $s=0$, both one-sided layers are zero.
For $s\ge1$, write
\[
 O_s=F_s-\binom{m+2}{s}\left(\frac{S_C}{m+2}\right)^{s/p}
\]
for the old layer and
\[
 O_s^\#=F_s^\#-\binom{m+1}{s}\left(\frac{S_C}{m+1}\right)^{s/p}
\]
for the new layer.
By \cref{cor:full-layers}, $F_s^\#\ge F_s$.

It remains to observe that, for fixed $S_C$, the defect
\[
 M\longmapsto \binom M s\left(\frac{S_C}{M}\right)^{s/p}=S_C^{s/p}\binom M s M^{-s/p}
\]
is increasing in the integer variable $M\ge s$.
Indeed, for $M>s$ the ratio of consecutive values is
\[
 \frac{M}{M-s}\left(\frac{M-1}{M}\right)^{s/p}.
\]
We claim this ratio is larger than $1$.
Since $0<1-1/M<1$ and $s/p<s$, we have
\[
 \left(1-\frac1M\right)^{s/p}>\left(1-\frac1M\right)^s.
\]
Bernoulli's inequality for the integer exponent $s$ gives
\[
 \left(1-\frac1M\right)^s\ge 1-\frac{s}{M}=\frac{M-s}{M}.
\]
Hence the displayed ratio is strictly larger than $1$.
Thus the defect is strictly increasing in $M$ on the integer range $M\ge s$.

If $s=m+2$, then the old all-$C$ defect is positive, whereas the new all-$C$ defect is zero because $\binom{m+1}{s}=0$.
If $1\le s\le m+1$, then both $m+1$ and $m+2$ lie in the range $M\ge s$, and the monotonicity just proved shows that the old defect is strictly larger than the new defect.
Hence the forbidden all-$C$ defect strictly decreases when $m+2$ is replaced by $m+1$, whenever $1\le s\le m+2$.
The lemma follows.
\end{proof}

\subsection{Pure cross layers}

The pure cross layer consists of the $s$-subsets which meet both local blocks.
This is the only local layer whose proof requires a more delicate argument.

\begin{lemma}\label{lem:stop-loss}
Let $0<\alpha<1$.
Suppose $A_i,B_j>0$ and $x_i,y_j>0$.
If
\[
 \sum_i A_i\min(T,x_i)\ge \sum_j B_j\min(T,y_j)\qquad\text{for every }T\ge0,
\]
then
\[
 \sum_i A_i x_i^\alpha\ge \sum_j B_j y_j^\alpha.
\]
\end{lemma}

\begin{proof}
For $x>0$ and $0<\alpha<1$,
\[
 x^\alpha=\alpha(1-\alpha)\int_0^\infty \min(T,x)T^{\alpha-2}\,dT.
\]
Indeed, splitting the integral at $T=x$ gives
\[
 \int_0^x T^{\alpha-1}\,dT+x\int_x^\infty T^{\alpha-2}\,dT
 =\frac{x^\alpha}{\alpha}+\frac{x^\alpha}{1-\alpha}.
\]
Multiplying by $\alpha(1-\alpha)$ gives the identity.
Integrating the assumed stop-loss inequality against the positive kernel $\alpha(1-\alpha)T^{\alpha-2}$ proves the claim.
\end{proof}
We shall use Descartes' rule of signs in the following standard form: a real polynomial whose nonzero coefficient sequence has at most one sign change has at most one positive zero, counted with multiplicity.

\begin{corollary}\label{coro:Descartes-rule}
Let
\[
f(x)=a_0+a_1x+\cdots+a_d x^d
\]
be a real polynomial.
Suppose that, after deleting zero coefficients, the sequence
\[
a_0,a_1,\ldots,a_d
\]
has at most one sign change.
If $f(x)>0$ for all sufficiently small positive $x$, and $f(c)>0$ for some $c>0$, then
\[
f(x)>0\qquad\text{for every }0<x\le c.
\]
\end{corollary}

\begin{proof}
By Descartes' rule of signs, $f$ has at most one positive zero, counted with multiplicity.
Suppose for a contradiction that $f(t)=0$ for some $t\in(0,c]$.
Since $f$ is positive near $0$ and also $f(c)>0$, this zero cannot be the only positive zero with odd multiplicity; otherwise the sign of $f$ would change and would have to change back before reaching $c$, giving another positive zero.
If the zero at $t$ has even multiplicity, then it is already counted with multiplicity at least $2$, again contradicting Descartes' rule.
Hence no such $t$ exists, and $f(x)>0$ for all $0<x\le c$.
\end{proof}

In the gap-two smoothing step, the old local blocks have sizes
\[
|C|=m+2,\qquad |B|=m,
\]
and total $p$-masses
\[
S_C,\qquad S_B.
\]
Thus the old pure cross layer of rank $s$ equals
\[
X_s^{\mathrm{old}}
=
\sum_{j=1}^{s-1}
\binom{m+2}{j}\binom{m}{s-j}
\left(\frac{S_C}{m+2}\right)^{j/p}
\left(\frac{S_B}{m}\right)^{(s-j)/p}.
\]
After smoothing, the new blocks satisfy
\[
|C^\#|=|B^\#|=m+1,
\]
with the same total $p$-masses $S_C$ and $S_B$.
Hence the new pure cross layer is
\[
X_s^{\mathrm{new}}
=
\sum_{j=1}^{s-1}
\binom{m+1}{j}\binom{m+1}{s-j}
\left(\frac{S_C}{m+1}\right)^{j/p}
\left(\frac{S_B}{m+1}\right)^{(s-j)/p}.
\]
\Cref{lem:pure-cross} compares $X_s^{\mathrm{old}}$ and $X_s^{\mathrm{new}}$, proving that
\[
X_s^{\mathrm{new}}\ge X_s^{\mathrm{old}},
\]
with strict inequality whenever the old pure cross layer is nonzero.

\begin{lemma}\label{lem:pure-cross}
Let $s\ge2$.
Under the local hypothesis \eqref{eq:rho-assumption}, the pure cross layer of rank $s$ does not decrease under the gap-two smoothing.
If the old pure cross layer is nonzero, then the inequality is strict.
\end{lemma}

\begin{proof}
By homogeneity set $S_C=1$, write $\rho=S_B/S_C$, and put $\alpha=1/p\in(0,1)$.


If $m\le s-2$, then $\binom m s=\binom{m+1}{s}=0$.
Thus the small-block monochromatic defect is zero both before and after smoothing.
The pure cross layer equals the full layer minus the large-block monochromatic defect.
The full layer does not decrease by \cref{cor:full-layers}, and the large-block defect decreases by the monotonicity argument in \cref{lem:one-sided}.
Hence the non-strict claim holds in this case.

We also record strictness in this subcase.
If $m=s-2$, then the old large-block defect is positive whereas the new large-block defect is zero, and the pure cross layer increases strictly.
If $m<s-2$ and the old pure cross layer is nonzero, then necessarily $1\le s\le2m+2$ and both local $p$-mass vectors are positive.
The old local $p$-mass vector strictly majorizes the new one and is not a permutation of it.
Hence the full $s$-layer strictly increases by \cref{lem:standard-schur-elementary}.
Since both large-block defects are zero when $m<s-2$, the pure cross layer also strictly increases.
Thus the lemma is proved when $m\le s-2$.

It remains to assume $m\ge s-1$.
Write
\[
 c:=m+1.
\]
The old sizes are $(c+1,c-1)$ and the new sizes are $(c,c)$, with $c\ge s$.
We first reduce to the boundary value
\[
 \rho_0:=\frac{c}{c+1}.
\]
Let $X=\rho^\alpha$.
The pure-layer difference is a polynomial
\[
 \Delta(X)=\sum_{u=1}^{s-1} C_u X^{s-u},
\]
where
\[
 C_u=\binom c u\binom c{s-u}c^{-\alpha s}
 -\binom{c+1}{u}\binom{c-1}{s-u}(c+1)^{-\alpha u}(c-1)^{-\alpha(s-u)}.
\]
For indices for which the old coefficient is nonzero, let $R_u$ be the ratio of the new coefficient to the old coefficient.
Then
\[
 R_u=\frac{c+1-u}{c+1}\cdot \frac{c}{c-s+u}\left(\frac{(c+1)^u(c-1)^{s-u}}{c^s}\right)^\alpha.
\]
We first note that $R_1>1$.
If
\[
 W:=\frac{(c+1)(c-1)^{s-1}}{c^s}\ge1,
\]
then
\[
 R_1\ge \frac{c}{c+1}\cdot \frac{c}{c-s+1}>1.
\]
If $W<1$, then $W^\alpha>W$ and hence
\[
 R_1>\frac{c}{c+1}\cdot\frac{c}{c-s+1}\cdot W
 =\frac{(c-1)^{s-1}}{c^{s-2}(c-s+1)}.
\]
The last expression is at least $1$: indeed this is equivalent to
\[
 \left(1-\frac1c\right)^{s-1}\ge1-\frac{s-1}{c},
\]
which follows from Bernoulli's inequality.
Therefore the preceding strict inequality gives $R_1>1$.

Moreover,
\[
 \frac{R_{u+1}}{R_u}
 =\frac{c-u}{c+1-u}\cdot\frac{c-s+u}{c-s+u+1}\left(\frac{c+1}{c-1}\right)^\alpha<1.
\]
Indeed, the first two factors are at most $c/(c+1)$ and $(c-1)/c$, respectively, and therefore the product is at most
\[
 \frac{c-1}{c+1}\left(\frac{c+1}{c-1}\right)^\alpha
 =\left(\frac{c+1}{c-1}\right)^{\alpha-1}<1.
\]
Thus the coefficient sequence $C_1,\ldots,C_{s-1}$ has at most one sign change, with direction from positive to negative.
The coefficient sequence of $\Delta(X)$, read in ascending powers of $X$, is the reverse of $C_1,\ldots,C_{s-1}$.
Hence it also has at most one sign change.
Hence, by Descartes' rule of signs, $\Delta$ has at most one positive zero.

It remains to prove the boundary inequality.
At $\rho=\rho_0$, set
\[
 A_u=\binom c u\binom c{s-u},
 \qquad
 B_u=\binom{c+1}{u}\binom{c-1}{s-u},
\]
\[
 x_u=c^{-u}(c+1)^{-(s-u)},
 \qquad
 y_u=(c+1)^{-u}\left(\frac{c}{(c+1)(c-1)}\right)^{s-u}.
\]
By \cref{lem:stop-loss}, it suffices to prove
\[
 H(T):=\sum_{u=1}^{s-1}A_u\min(T,x_u)-\sum_{u=1}^{s-1}B_u\min(T,y_u)\ge0
 \qquad (T\ge0).
\]
Since $y_{u+1}/y_u=(c-1)/c<1$, the $y_u$ are decreasing.
On every interval whose endpoints are consecutive $y$-values, the second sum is affine in $T$, while the first sum is concave in $T$.
Therefore $H$ is concave on each such interval, also $H(0)=0$, and it is enough to check $T=y_j$ for $1\le j\le s-1$, together with $T=\infty$.

We next show that the check at $T=y_1$ is the same as the check at infinity.
For $1\le u\le s-1$,
\[
 \frac{x_u}{y_1}=\left(\frac{c+1}{c}\right)^u\left(\frac{c-1}{c}\right)^{s-1}.
\]
The right-hand side is increasing in $u$, and so
\[
 \frac{x_u}{y_1}\le \left(\frac{c+1}{c}\right)^{s-1}\left(\frac{c-1}{c}\right)^{s-1}
 =\left(1-\frac1{c^2}\right)^{s-1}<1.
\]
Thus $x_u<y_1$ for every $u$.
Hence
\[
 \min(y_1,x_u)=x_u,
 \qquad
 \min(y_1,y_u)=y_u,
\]
for all $u$, and consequently
\[
 H(y_1)=\sum_{u=1}^{s-1}A_ux_u-\sum_{u=1}^{s-1}B_uy_u=H(\infty).
\]
It is therefore enough to check $T=y_j$ for $1\le j\le s-1$.

Put
\[
 Y_j:=\frac{H(y_j)}{y_j},
 \qquad
 q:=\frac{c-1}{c},
 \qquad
 \beta:=1-\frac1{c^2}.
\]
A direct calculation gives
\[
 \Theta_{u,j}:=\frac{x_u}{y_j}=\left(\frac{c+1}{c}\right)^u\left(\frac{c-1}{c}\right)^{s-j}.
\]
For fixed $j$, the quantity $\Theta_{u,j}$ is strictly increasing in $u$.
Moreover,
\[
 \Theta_{s-j,j}=\left(1-\frac1{c^2}\right)^{s-j}<1.
\]
If $j\ge2$, then $s-j\le s-2\le c-2$, and Bernoulli's inequality gives
\[
 \Theta_{s-j+1,j}=\frac{c+1}{c}\left(1-\frac1{c^2}\right)^{s-j}
 \ge \frac{c+1}{c}\left(1-\frac{s-j}{c^2}\right)
 \ge \frac{c+1}{c}\left(1-\frac{c-2}{c^2}\right)>1.
\]
If $j=1$, then $s-j+1=s$, which is outside the range $1\le u\le s-1$; in this case only the inequality $\Theta_{u,1}<1$ for $u\le s-1$ is needed.
Hence the truncation point is $s-j$ for every $j$.

We now derive the formula for $Y_j$ explicitly.
At $T=y_j$, the $B$-side contribution is
\[
 \sum_{u=1}^jB_u+\sum_{u=j+1}^{s-1}B_uq^{u-j},
\]
because $y_u\ge y_j$ for $u\le j$ and $y_u/y_j=q^{u-j}$ for $u>j$.
On the $A$-side, the truncation gives
\[
 \sum_{u=1}^{s-j}A_u\Theta_{u,j}+\sum_{u=s-j+1}^{s-1}A_u.
\]
In the first sum set $v=s-u$.
Since $A_u=A_{s-u}$ and
\[
 \Theta_{u,j}=q^{s-u-j}\beta^u,
\]
we obtain
\[
 \sum_{u=1}^{s-j}A_u\Theta_{u,j}=\sum_{v=j}^{s-1}A_vq^{v-j}\beta^{s-v}.
\]
Similarly,
\[
 \sum_{u=s-j+1}^{s-1}A_u=\sum_{v=1}^{j-1}A_v.
\]
Combining these two re-indexed sums with the $B$-side contribution yields
\begin{equation}\label{eq:Yj-formula}
 Y_j=\sum_{u=1}^{j-1}(A_u-B_u)+\sum_{u=j}^{s-1}q^{u-j}\left(A_u\beta^{s-u}-B_u\right).
\end{equation}
Let $d_u=A_u-B_u$ and $D_i=\sum_{u=1}^i d_u$.
In the present gap-two situation, by a telescoping/Vandermonde identity,
\[
 D_i=\binom c i\binom{c-1}{s-i-1}-\binom{c-1}{s-1}=\frac{s-i}{c}A_i-K,
 \qquad
 K:=\binom{c-1}{s-1}.
\]
Write $Y_j=B_j^{*}-L_j$, where Abel summation gives
\[
 B_j^{*}=D_{s-1}q^M+(1-q)\sum_{i=j}^{s-2}D_iq^{i-j},
 \qquad
 M:=s-1-j,
\]
and
\[
 L_j=\sum_{u=j}^{s-1}q^{u-j}A_u(1-\beta^{s-u}).
\]
Since $1-\beta^{s-u}\le (s-u)/c^2$,
\[
 L_j\le \frac1{c^2}\sum_{u=j}^{s-1}(s-u)A_uq^{u-j}.
\]
Substituting the displayed formula for $D_i$ into $B_j^{*}$, all terms involving $(s-u)A_u$ cancel except the final one.
Therefore
\[
 B_j^{*}-\frac1{c^2}\sum_{u=j}^{s-1}(s-u)A_uq^{u-j}
 =q^M\left(D_{s-1}-\frac{A_{s-1}}{c^2}\right)-\frac Kc\sum_{i=j}^{s-2}q^{i-j}.
\]
Now
\[
 D_{s-1}=\binom{c-1}{s-2},
 \qquad
 \frac{A_{s-1}}{c^2}=\frac1{s-1}\binom{c-1}{s-2},
\]
\[
 K=\binom{c-1}{s-1}=\frac{c-s+1}{s-1}\binom{c-1}{s-2},
\]
and
\[
 \sum_{i=j}^{s-2}q^{i-j}=c(1-q^M).
\]
Consequently,
\[
 Y_j\ge \frac{\binom{c-1}{s-2}}{s-1}\left((s-2)q^M-(c-s+1)(1-q^M)\right).
\]
The bracket is nonnegative precisely when
\[
 (c-1)q^M\ge c-s+1.
\]
But $0\le M\le s-2$ and $q=1-1/c$, so Bernoulli's inequality gives
\[
 q^M\ge1-\frac Mc\ge\frac{c-s+2}{c}\ge\frac{c-s+1}{c-1}.
\]
Thus $Y_j\ge0$ for all $j$, and hence $H(T)\ge0$ for all $T\ge0$.

If $s=2$, the pure cross-layer before smoothing is
\[
 m(m+2)\left(\frac{S_C}{m+2}\right)^{1/p}\left(\frac{S_B}{m}\right)^{1/p},
\]
whereas after smoothing it is
\[
 (m+1)^2\left(\frac{S_C}{m+1}\right)^{1/p}\left(\frac{S_B}{m+1}\right)^{1/p}.
\]
Since $(m+1)^2>m(m+2)$ and $1-1/p>0$, the latter is strictly larger.
For $s\ge3$, the lower bound above is strict at $j=s-1$.
Thus the boundary inequality gives
\[
 \Delta(\rho_0^\alpha)>0
\]
whenever the old pure cross-layer is nonzero.

We now explain why this strictness propagates from the boundary point to every admissible value of $\rho$.
Put $X_0=\rho_0^\alpha$.
We have already shown that the coefficient sequence of $\Delta$ has at most one sign change and that its leading nonzero coefficient is positive.
Hence $\Delta(X)>0$ for all sufficiently large $X>0$, and Descartes' rule implies that $\Delta$ has at most one positive zero, counted with multiplicity.
Since $\Delta(X_0)>0$, there can be no zero on $[X_0,\infty)$: a simple zero would force the polynomial to change sign and then require a second positive zero before becoming positive again for large $X$, while a zero without sign change would have even multiplicity and hence would already count as at least two positive zeros.
Therefore
\[
 \Delta(X)>0\qquad\text{for every }X\ge X_0.
\]
Since $X=\rho^\alpha$ and $\rho\ge\rho_0$ imply $X\ge X_0$, the pure cross-layer inequality is strict for every admissible $\rho$ whenever the old pure cross-layer is nonzero.
In particular, the desired non-strict inequality also holds for every $\rho\ge\rho_0$.
\end{proof}


\subsection{Proof of Theorem~\ref{thm:main} for $p > 1$}

We now complete the proof of Theorem~\ref{thm:main} for the case $p > 1$. 

\begin{proposition}\label{prop:embedded-smoothing}
Let $r\ge4$, $k\ge2$, $p>1$, and $n>(r-1)k$.
Let $\mathbf n=(n_1,\ldots,n_k)$ be an extremal tuple for $\Lambda_p$, ordered so that $n_1\ge\cdots\ge n_k$.
Then $n_1-n_k\le1$.
\end{proposition}

\begin{proof}
Assume, toward a contradiction, that
\[
 n_1-n_k\ge2.
\]
Let
\[
 A:=V_1,
 \qquad
 B:=V_k,
 \qquad
 \abs A=a,
 \qquad
 \abs B=b.
\]
Let the old eigenvector take the values
\[
 \alpha:=a_1\quad\text{on }A,
 \qquad
 \beta:=a_k\quad\text{on }B.
\]
By \cref{prop:structural-hypotheses}, $\alpha\le\beta$ and
\begin{equation}\label{eq:embedded-structural-bound}
 b\beta^p\ge (a-1)\alpha^p.
\end{equation}
Choose a subset
\[
 C\subset A,
 \qquad
 \abs C=b+2,
\]
and put
\[
 E:=A\setminus C,
 \qquad
 \abs E=a-b-2.
\]
We smooth only the local pair $(C,B)$.
Choose one vertex $v\in C$ and define
\[
 C^\#:=C\setminus\{v\},
 \qquad
 B^\#:=B\cup\{v\}.
\]
Thus $\abs{C^\#}=\abs{B^\#}=b+1$.
The new color classes are
\[
 V_1^\#:=E\cup C^\#,
 \qquad
 V_k^\#:=B^\#,
\]
and all middle classes are unchanged.
Let $G^\#$ be the complete $k$-chromatic $r$-graph with these new classes.

Define a comparison vector $y$ by keeping the old values on $E$ and on all middle classes, while setting
\[
 y_i=\alpha^\#:=\left(\frac{(b+2)\alpha^p}{b+1}\right)^{1/p}\quad (i\in C^\#),
\]
and
\[
 y_i=\beta^\#:=\left(\frac{b\beta^p}{b+1}\right)^{1/p}\quad (i\in B^\#).
\]
Then
\[
 (b+1)(\alpha^\#)^p+(b+1)(\beta^\#)^p=(b+2)\alpha^p+b\beta^p,
\]
so $\norm y_p=\norm x_p=1$.
Notice that $y$ need not be constant on the new class $V_1^\#=E\cup C^\#$; this is harmless, since $y$ is only a test vector for $G^\#$.

The local masses of $C$ and $B$ are
\[
 S_C=(b+2)\alpha^p,
 \qquad
 S_B=b\beta^p.
\]
The structural bound \eqref{eq:embedded-structural-bound} gives
\[
 \frac{S_B}{S_C}=\frac{b\beta^p}{(b+2)\alpha^p}
 \ge \frac{a-1}{b+2}\ge \frac{b+1}{b+2}.
\]
Thus the local gap-two hypotheses of \cref{sec:smoothing} are satisfied with $m=b$.

Let
\[
 R:=E\cup V_2\cup\cdots\cup V_{k-1}
\]
be the part of the vertex set not moved by the local smoothing.
We compare $P/r!$ by fixing a set
\[
 F\subseteq R,
 \qquad
 \abs F=f,
\]
and completing it with
\[
 s:=r-f
\]
vertices from $C\cup B$ before smoothing, or from $C^\#\cup B^\#$ after smoothing.
The product of the weights on $F$ is the same before and after smoothing, so it suffices to compare the local completion contribution for each fixed $F$.

There are three cases.

\smallskip \noindent\emph{Case 1: $F$ is not contained in $E$, and $F\ne\emptyset$.}
Then the fixed set $F$ already prevents monochromaticity after any nonempty local completion.
Indeed, either $F$ itself meets at least two fixed color classes, in which case the completed $r$-set is already non-monochromatic, or $F$ is contained in a single middle color class.
In the latter case every local vertex chosen from $C\cup B$ before smoothing, or from $C^\#\cup B^\#$ after smoothing, lies in one of the two extreme color classes and hence has a color different from that middle class.
Therefore, if $s>0$, every choice of $s$ local vertices is admissible and the local contribution is the full local $s$-layer, which does not decrease by \cref{cor:full-layers}.
If $s=0$, no local vertex is chosen; the contribution is either zero or the fixed monomial of $F$, and in either case it is independent of the local smoothing.

\smallskip \noindent\emph{Case 2: $\emptyset\ne F\subseteq E$.}
Here $F$ lies entirely in the old largest color class.
A local completion is forbidden exactly when all $s$ local vertices are chosen from the $C$-side.
Hence the local contribution is a one-sided layer, and it does not decrease by \cref{lem:one-sided}.
If $s=0$, both old and new contributions are zero.

\smallskip \noindent\emph{Case 3: $F=\emptyset$.}
Then the whole edge lies in the local pair.
The admissible local $r$-sets are exactly those meeting both local blocks.
Hence the local contribution is the pure cross layer of rank $r$, and \cref{lem:pure-cross} shows that it does not decrease.

Summing over all $F\subseteq R$ gives
\[
 P_{G^\#}(y)\ge P_Q(x).
\]
In fact the inequality is strict.
By \cref{lem:no-small-class-unified}, the smallest class has size $b\ge r-1$.
Hence the old pure rank-$r$ cross layer is nonzero; for example, it contains the contribution obtained by choosing one vertex from $C$ and $r-1$ vertices from $B$.
Therefore Case 3 is strict by \cref{lem:pure-cross}.
Consequently,
\[
 P_{G^\#}(y)>P_Q(x)=\lambda^{(p)}(Q(\mathbf n)).
\]
Since $\norm y_p=1$, the new tuple has larger $p$-spectral radius, contradicting the extremality of $\mathbf n$.
Hence $n_1-n_k\le1$.
\end{proof}

\section{Proof of Theorem~\ref{thm:main} for $p=1$}
\label{sec:pone}

In this section, we present the proof of Theorem~\ref{thm:main} for the case $p = 1$. 

Throughout this section write
\[
 \lambda(G):=\lambda^{(1)}(G)=\max_{x_v\ge0,\,\sum_v x_v=1}P_G(x).
\]
The argument in this section works for every $r\ge3$.

Note that it suffices to prove the following proposition. 

\begin{proposition}\label{prop:pone-balancing}
Let $r\ge3$, $k\ge2$, and $n>(r-1)k$.
If $\mathbf n=(n_1,\ldots,n_k)$ maximizes $\Lambda_1(\mathbf n)=\lambda(Q(\mathbf n))$, then $\mathbf n$ is balanced; that is, $\abs{n_i-n_j}\le1$ for all $i,j$.
\end{proposition}

The proof of \cref{prop:pone-balancing} is divided into several lemmas. 

We shall also use the following elementary form of Schur concavity.

\begin{lemma}\label{lem:pone-robin}
Let $u=(\beta,\ldots,\beta,\alpha,\ldots,\alpha)$ have $b$ entries equal to $\beta$ and $a$ entries equal to $\alpha$, where $0\le\alpha\le\beta$.
If $v$ is any vector of length $a+b$, all of whose entries lie in $[\alpha,\beta]$, and if $\sum_i v_i=\sum_i u_i$, then
\[
 e_q(v)\ge e_q(u)
\]
for every $q\ge2$.
\end{lemma}

\begin{proof}
Since $u$ majorizes $v$, the claim follows from the Schur-concavity of $e_q$ on the nonnegative orthant.
\end{proof}

For a $k$-tuple $\mathbf n=(n_1,\ldots,n_k)$ of nonnegative integers summing to $n$, put
\[
 \Lambda_1(\mathbf n)=\lambda(Q(n_1,\ldots,n_k)).
\]
Choose $\mathbf n$ so that $\Lambda_1(\mathbf n)$ is maximal among all such $k$-tuples.

Now suppose that a maximizing tuple is not balanced.
Choose two classes $A,B$ with
\[
 \abs A=a,
 \qquad
 \abs B=b,
 \qquad
 a\ge b+2.
\]
By \cref{lem:no-small-class-unified}, $b\ge r-1$.
Let $x$ be a nonnegative eigenvector to $\lambda(Q(\mathbf n))$.
By \cref{lem:quoted-complete-chromatic-positivity}, $x$ is positive and is constant on each class.
Write
\[
 x_v=\alpha\quad(v\in A),
 \qquad
 x_v=\beta\quad(v\in B).
\]


The following lemma provides a necessary optimality inequality at $p=1$. 

\begin{lemma}\label{lem:pone-necessary}
Let $r\ge3$, $k\ge2$, and $n>(r-1)k$.
Let $\mathbf n=(n_1,\ldots,n_k)$ maximize $\Lambda_1$.
Let $A,B$ be two classes of $Q(\mathbf n)$ with sizes $a,b$, where $a\ge b+2$.
Let a Lagrangian maximizer take the values $\alpha$ on $A$ and $\beta$ on $B$.
Then
\[
        \binom{a-1}{r-1}\alpha^{r-1}
        \le
        \binom{b}{r-1}\beta^{r-1}.
\]
Consequently,
\[
        \frac{\alpha}{\beta}<\frac{b}{b+1}.
\]
\end{lemma}

\begin{proof}
Move one vertex $v$ from $A$ to $B$ and keep the old vector $x$.
By maximality, the value cannot increase.
As in the proof of \cref{lem:no-small-class-unified}, the only changed edges are those of the form $\{v\}\cup S$ with $S\in(A\setminus\{v\})^{(r-1)}$, which are gained, and those of the form $\{v\}\cup T$ with $T\in B^{(r-1)}$, which are lost.
Therefore
\[
 \binom{a-1}{r-1}\alpha^r\le \binom b{r-1}\alpha\beta^{r-1}.
\]
Since $\alpha>0$, this gives the first assertion.

Let $t=\alpha/\beta$.
Since $a-1\ge b+1$,
\[
 t^{r-1}\le \frac{\binom b{r-1}}{\binom{a-1}{r-1}}
 \le \frac{\binom b{r-1}}{\binom{b+1}{r-1}}
 =\frac{b-r+2}{b+1}.
\]
For $r\ge3$, Bernoulli's inequality gives
\[
 \frac{b-r+2}{b+1}=1-\frac{r-1}{b+1}<\left(1-\frac1{b+1}\right)^{r-1}=\left(\frac b{b+1}\right)^{r-1}.
\]
Thus $t<b/(b+1)$.
\end{proof}

We now define the balancing move with a new test vector.
Replace the pair of class sizes $(a,b)$ by $(a-1,b+1)$, and redistribute the total $\ell_1$-mass of each of the two classes uniformly over its new class.
Thus the new vertex weights are
\[
 \alpha^+=\frac{a\alpha}{a-1},
 \qquad
 \beta^- =\frac{b\beta}{b+1},
\]
on the new $A$-class and new $B$-class respectively, while all other class weights remain unchanged.
The $\ell_1$-constraint is preserved because
\[
 (a-1)\alpha^+ +(b+1)\beta^-=a\alpha+b\beta.
\]
Let
\[
 U=\{\alpha,\ldots,\alpha\}_{a\text{ times}}\cup \{\beta,\ldots,\beta\}_{b\text{ times}},
\]
and
\[
 U^+=\{\alpha^+,\ldots,\alpha^+\}_{a-1\text{ times}}\cup \{\beta^-,\ldots,\beta^-\}_{b+1\text{ times}}.
\]
Let $R$ be the multiset of all weights outside $A\cup B$, and put $E_j=e_j(R)$, with the conventions $E_0=1$ and $E_j=0$ if $j>\abs R$.

The change in the polynomial form, divided by $r!$, is
\begin{equation}\label{eq:pone-decomp}
 \frac{P_{\rm new}-P_{\rm old}}{r!}
 =\sum_{q=2}^{r-1}E_{r-q}\bigl(e_q(U^+)-e_q(U)\bigr)+\bigl(m_r(U^+)-m_r(U)\bigr),
\end{equation}
where
\[
 m_r(U)=\sum_{s=1}^{r-1}\binom as\binom b{r-s}\alpha^s\beta^{r-s}
\]
is the contribution of the $r$-sets using vertices from $A\cup B$ and meeting both classes.

Indeed, for $1\le q\le r-1$, every $r$-set using $q$ vertices from $A\cup B$ and $r-q$ vertices from outside $A\cup B$ is an edge both before and after the move, and these terms contribute $E_{r-q}(e_q(U^+)-e_q(U))$.
The $q=1$ term cancels because $e_1(U^+)=e_1(U)$, while the $q=0$ contribution is unchanged.
For $q=r$, one must exclude the monochromatic $r$-sets inside $A$ or inside $B$, leaving precisely $m_r$.

The nonnegativity of the first sum in \eqref{eq:pone-decomp} follows from \cref{lem:pone-robin}.
Indeed, \cref{lem:pone-necessary} gives
\[
 \alpha<\beta^- =\frac{b\beta}{b+1}.
\]
Also, since $a\ge b+2$,
\[
 \frac b{b+1}<\frac{a-1}{a},
\]
and hence
\[
 \alpha^+=\frac{a\alpha}{a-1}<\beta.
\]
Thus every entry of $U^+$ lies in $[\alpha,\beta]$, and $U^+$ has the same total sum as $U$.
\Cref{lem:pone-robin} yields
\begin{equation}\label{eq:pone-external-nonnegative}
 e_q(U^+)\ge e_q(U),
 \qquad
 2\le q\le r-1.
\end{equation}
It remains to prove that the mixed two-class term strictly increases.

\begin{lemma}\label{lem:pone-mixed}
Let $r\ge3$, $a\ge b+2$, $b\ge r-1$, and $\alpha,\beta>0$.
Suppose
\[
 \frac\alpha\beta\le \frac b{b+1}.
\]
Set
\[
 \alpha^+=\frac{a\alpha}{a-1},
 \qquad
 \beta^- =\frac{b\beta}{b+1}.
\]
Then
\[
 \sum_{s=1}^{r-1}\binom{a-1}{s}\binom{b+1}{r-s}(\alpha^+)^s(\beta^-)^{r-s}
 >
 \sum_{s=1}^{r-1}\binom as\binom b{r-s}\alpha^s\beta^{r-s}.
\]
\end{lemma}

\begin{proof}
By homogeneity, set $\beta=1$ and write
\[
 t=\alpha,
 \qquad
 c=\frac b{b+1}.
\]
It is enough to prove the assertion for $0<t\le c$.
Let $\Delta(t)$ be the left side minus the right side.
Then
\[
 \Delta(t)=\sum_{s=1}^{r-1}D_s t^s,
\]
where
\[
 D_s=\binom{a-1}{s}\binom{b+1}{r-s}\left(\frac a{a-1}\right)^s\left(\frac b{b+1}\right)^{r-s}
 -\binom as\binom b{r-s}.
\]
Define
\[
 R_s=\frac{\binom{a-1}{s}\binom{b+1}{r-s}\left(\frac a{a-1}\right)^s\left(\frac b{b+1}\right)^{r-s}}{\binom as\binom b{r-s}}
 \qquad (1\le s\le r-1).
\]
Then $D_s$ has the sign of $R_s-1$.
A direct simplification gives
\[
 R_s=\frac{a-s}{a}\left(\frac a{a-1}\right)^s\frac{b+1}{b+1-r+s}\left(\frac b{b+1}\right)^{r-s}.
\]
For $1\le s\le r-2$,
\[
 \frac{R_{s+1}}{R_s}=\frac{a(a-s-1)}{(a-1)(a-s)}\cdot \frac{(b+1)(b+1-r+s)}{b(b+2-r+s)}<1.
\]
Indeed, the first factor is strictly smaller than $1$ because the denominator minus the numerator is $s>0$, and the second factor is strictly smaller than $1$ because the denominator minus the numerator is $r-s-1>0$.
Hence $R_1,\ldots,R_{r-1}$ is strictly decreasing.

Moreover, Bernoulli's inequality gives
\[
 R_1=\frac{b+1}{b-r+2}\left(\frac b{b+1}\right)^{r-1}>1
\]
and
\[
 R_{r-1}=\frac{a-r+1}{a}\left(\frac a{a-1}\right)^{r-1}<1.
\]
Thus the coefficient sequence $D_1,\ldots,D_{r-1}$ has exactly one sign change, ignoring possible zero coefficients.
By Descartes' rule of signs, the polynomial
\[
 \frac{\Delta(t)}{t}=D_1+D_2t+\cdots+D_{r-1}t^{r-2}
\]
has at most one positive zero, counted with multiplicity.
Since $D_1>0$, $\Delta(t)>0$ for all sufficiently small positive $t$.
Therefore it remains to prove
\begin{equation}\label{eq:pone-delta-c}
 \Delta(c)>0.
\end{equation}

At $t=c$, the two multisets are
\[
 U=\{c,\ldots,c\}_{a\text{ times}}\cup\{1,\ldots,1\}_{b\text{ times}}
\]
and
\[
 U^+=\left\{\frac{ac}{a-1},\ldots,\frac{ac}{a-1}\right\}_{a-1\text{ times}}
 \cup\{c,\ldots,c\}_{b+1\text{ times}}.
\]
Because $a\ge b+2$, we have
\[
 c<\frac{ac}{a-1}<1.
\]
Thus every entry of $U^+$ lies in the interval $[c,1]$.
Moreover,
\[
        \sum_{u\in U}u
        =
        ac+b,
\]
whereas, using $c=b/(b+1)$,
\[
        \sum_{u\in U^+}u
        =
        (a-1)\frac{ac}{a-1}+(b+1)c
        =
        ac+(b+1)c
        =
        ac+b.
\]
Hence $U$ and $U^+$ have the same total sum.
Applying Lemma~\ref{lem:pone-robin} with $\alpha=c$, $\beta=1$, $u=U$, and $v=U^+$, we obtain
\begin{equation}\label{eq:pone-er}
 e_r(U^+)\ge e_r(U).
\end{equation}

Now compare the monochromatic $r$-subsets within the two classes.
Define, for real $x\ge r$,
\[
 L_r(x)=\binom xr-\binom{x-1}{r}\left(\frac x{x-1}\right)^r.
\]
We claim that $L_r$ is strictly increasing on $[r,\infty)$.
Indeed,
\[
 L_r(x)=A_r(x)B_r(x),
\]
where
\[
 A_r(x)=\frac{x(x-2)(x-3)\cdots(x-r+1)}{r!(x-1)^{r-1}}
\]
and
\[
 B_r(x)=(x-1)^r-(x-r)x^{r-1}.
\]
The logarithmic derivative of $A_r$ is
\[
 \frac{A_r'(x)}{A_r(x)}=\frac1x+\sum_{j=2}^{r-1}\frac1{x-j}-\frac{r-1}{x-1}.
\]
This is positive for $x\ge r$: indeed,
\[
 \frac1x+\frac1{x-2}-\frac2{x-1}=\frac{2}{x(x-1)(x-2)}>0,
\]
and, for $j\ge3$, each term $1/(x-j)-1/(x-1)$ is positive.
Hence $A_r$ is strictly increasing and positive.

For $B_r$, put $x=r+y$, $y\ge0$.
Then
\[
 B_r(r+y)=(r-1+y)^r-y(r+y)^{r-1}.
\]
In this polynomial in $y$, the coefficient of $y^j$ for $1\le j\le r-2$ is
\[
 \binom rj(r-1)^{r-j}-\binom{r-1}{j-1}r^{r-j}
 =\binom{r-1}{j-1}r^{r-j}\left(\frac rj\left(1-\frac1r\right)^{r-j}-1\right)>0.
\]
The strict positivity follows from Bernoulli's inequality, since for $1\le j\le r-2$,
\[
 \left(1-\frac1r\right)^{r-j}>1-\frac{r-j}{r}=\frac jr,
\]
and hence $\frac rj(1-1/r)^{r-j}>1$.
The constant term is positive, and the coefficients of $y^{r-1}$ and $y^r$ are zero.
Hence $B_r$ is positive and strictly increasing on $[r,\infty)$.
Therefore $L_r=A_rB_r$ is strictly increasing.

Let $M$ and $M^+$ denote the monochromatic contributions at $t=c$ before and after the balancing move.
Then
\begin{align*}
 M-M^+
&=\left[\binom ar c^r+\binom br\right]
 -\left[\binom{a-1}{r}\left(\frac{ac}{a-1}\right)^r+\binom{b+1}{r}c^r\right]\\
&=c^r\bigl(L_r(a)-L_r(b+1)\bigr)>0,
\end{align*}
since $a\ge b+2$.
Combining this strict inequality with \eqref{eq:pone-er}, we get
\[
 \Delta(c)=\bigl(e_r(U^+)-e_r(U)\bigr)+(M-M^+)>0.
\]
This proves \eqref{eq:pone-delta-c}.

We now return to the Descartes step.
Since $\Delta(t)/t$ has at most one positive zero counted with multiplicity, it has no zero in $(0,c]$.
Indeed, a sign-changing zero in $(0,c)$ would force a second positive zero before reaching the positive value at $c$, while a non-sign-changing interior zero would have even multiplicity and hence would already count at least twice.
Also $\Delta(c)>0$, so $c$ itself is not a zero.
Therefore $\Delta(t)>0$ for every $0<t\le c$, completing the proof.
\end{proof}

We now complete the proof of \cref{prop:pone-balancing}. 

\begin{proof}[Proof of \cref{prop:pone-balancing}]
Assume that a maximizing tuple is not balanced.
Choose two classes $A,B$ of sizes $a,b$ with $a\ge b+2$.
By \cref{lem:no-small-class-unified}, $b\ge r-1$.
Let $\alpha,\beta$ be the corresponding class weights in a nonnegative eigenvector.
\Cref{lem:pone-necessary} gives $\alpha/\beta<b/(b+1)$.

Perform the balancing move $(a,b)\mapsto(a-1,b+1)$ and use the redistributed test vector described above.
By \eqref{eq:pone-decomp} and \eqref{eq:pone-external-nonnegative}, all terms in the external sum are nonnegative.
By \cref{lem:pone-mixed}, the mixed two-class term is strictly positive.
Thus
\[
 P_{\rm new}>P_{\rm old}=\lambda(Q(\mathbf n)),
\]
so the new complete $k$-chromatic graph has Lagrangian larger than $\Lambda_1(\mathbf n)$, contradicting the maximality of $\mathbf n$.
Hence every maximizing tuple is balanced.
\end{proof}







\section{Proof of Theorem~\ref{thm:evaluation-Qkr}}
\label{sec:evaluation-Qkr}

In this section, we present the proof of Theorem~\ref{thm:evaluation-Qkr}. 
We write $[z^r]f(z)$ for the coefficient of $z^r$ in $f(z)$.

\begin{lemma}
\label{lem:edge-count-Qkr}
Let $r\ge3$, $k\ge2$, and $n>(r-1)k$.
Write
\[
        n=kq+t,\qquad 0\le t<k.
\]
Then
\[
        |Q_k^r(n)|
        =
        \binom nr
        -
        t\binom{q+1}{r}
        -
        (k-t)\binom q r.
\]
Moreover,
\[
        |Q_k^r(n)|
        \le
        \binom nr-k\binom{n/k}{r},
\]
with equality if and only if $t=0$.
\end{lemma}

\begin{proof}
The non-edges of $Q_k^r(n)$ are precisely the $r$-sets contained in a single colour class, giving the stated formula for $|Q_k^r(n)|$.

Since $n/k>r-1$, the function $x\mapsto \binom{x}{r}$ is strictly convex on the interval relevant to us.
Therefore
\[
        t\binom{q+1}{r}+(k-t)\binom q r
        \ge
        k\binom{q+t/k}{r}
        =
        k\binom{n/k}{r}.
\]
The inequality is strict when $0<t<k$.
\end{proof}

By \cref{lem:quoted-holder-reduction} and \cref{lem:edge-count-Qkr}, \cref{thm:evaluation-Qkr} is reduced to the following proposition. 

\begin{proposition}
\label{prop:pone-evaluation-Qkr}
Let $r\ge3$, $k\ge2$, and $n>(r-1)k$.
Then
\[
        \lambda^{(1)}(Q_k^r(n))
        \le
        r!\left(\binom nr-k\binom{n/k}{r}\right)n^{-r}.
\]
If $k\mid n$, equality holds.
If $k\nmid n$, the inequality is strict.
\end{proposition}

\begin{proof}
Let $Q=Q_k^r(n)$, and let $V_1,\ldots,V_k$ be its colour classes.

By \cref{lem:quoted-basic-variational,lem:quoted-complete-chromatic-positivity}, an optimal vector may be chosen nonnegative and is constant on each colour class.

For a colour class $V_i$, let $b_i$ be its total weight.
Thus, after the preceding averaging, every vertex in $V_i$ has weight $b_i/|V_i|$, and
\[
        b_1+\cdots+b_k=1.
\]

We shall also use that among colour classes of the same size, the total class masses may be assumed equal.
Indeed, take two classes $U,V$ of the same size $m$, with total masses $b_U,b_V$, and replace these masses by
\[
        \bar b=\frac{b_U+b_V}{2}.
\]
All other class masses are fixed.
If an edge uses $s$ vertices from $U\cup V$, with $1\le s\le r-1$, and also uses at least one vertex outside $U\cup V$, then the relevant local contribution is the $s$-th elementary symmetric sum of the $2m$-vector
\[
        \underbrace{b_U/m,\ldots,b_U/m}_{m},
        \underbrace{b_V/m,\ldots,b_V/m}_{m}.
\]
For fixed total $b_U+b_V$, this is maximized when all $2m$ coordinates are equal.
If an edge uses all $r$ vertices from $U\cup V$, then its local contribution is
\[
        e_r(w)
        -
        \binom mr\left(\frac{b_U}{m}\right)^r
        -
        \binom mr\left(\frac{b_V}{m}\right)^r,
\]
where $w$ is the above $2m$-vector.
The first term does not decrease by Schur-concavity, while the subtracted term decreases by convexity of $x^r$.
Thus this averaging does not decrease $P_Q$.
Iterating proves the claim.

Write
\[
        n=kq+t,\qquad 0\le t<k.
\]
Since $n>(r-1)k$, we have $q\ge r-1$.

If $t=0$, all classes have the same size $q=n/k$.
The preceding equal-size class averaging shows that an optimal vector may be chosen with all class masses equal to $1/k$, hence all vertex weights equal to $1/n$.
Therefore
\[
        \lambda^{(1)}(Q_k^r(n))
        =
        r!\left(\binom nr-k\binom{n/k}{r}\right)n^{-r}.
\]
So assume henceforth that
\[
        1\le t\le k-1.
\]

Set
\[
        A=t(q+1),\qquad B=(k-t)q,\qquad n=A+B.
\]
Let $x\in[0,1]$ be the total mass on the $t$ classes of size $q+1$.
Then each large-class vertex has weight
\[
        a=\frac{x}{A},
\]
and each small-class vertex has weight
\[
        b=\frac{1-x}{B}.
\]
Define
\[
\begin{aligned}
        F(x)
        &:=
        \sum_{j=0}^{r}
        \binom Aj\binom B{r-j}
        \left(\frac{x}{A}\right)^j
        \left(\frac{1-x}{B}\right)^{r-j}  \\
        &\quad
        -
        t\binom{q+1}{r}\left(\frac{x}{A}\right)^r
        -
        (k-t)\binom q r\left(\frac{1-x}{B}\right)^r .
\end{aligned}
\]
Then
\[
        \lambda^{(1)}(Q_k^r(n))=r!\max_{0\le x\le1}F(x).
\]

Let
\[
        \tau=\frac tk,
        \qquad
        \xi=\frac{A}{n}
        =
        \frac{t(q+1)}{n}.
\]
The point $x=\tau$ corresponds to equal colour-class masses, while $x=\xi$ corresponds to equal vertex weights.

We first localize the maximum of $F$.

\medskip

\begin{claim}\label{claim:localize-maximum}
Every maximizer of $F$ lies in $[\tau,\xi]$.
\end{claim}

\smallskip

For $0<x<1$, put
\[
        R=\frac ba.
\]
Let
\[
        S_{r-2}(a,b)
        =
        \sum_{j=0}^{r-2}
        \binom{A-1}{j}\binom{B-1}{r-2-j}
        a^j b^{r-2-j}.
\]

A generating-function differentiation gives
\begin{equation}
        F'(x)
        =
        (b-a)S_{r-2}(a,b)
        -
        \binom q{r-1}a^{r-1}
        +
        \binom{q-1}{r-1}b^{r-1}.
\label{eq:derivative}
\end{equation}
Indeed, differentiating the all-$r$-set term gives
\[
\frac{d}{dx}
\sum_{j=0}^r
\binom Aj\binom B{r-j}a^j b^{r-j}
=
e_{r-1}(a^{A-1},b^B)-e_{r-1}(a^A,b^{B-1})
=
(b-a)S_{r-2}(a,b),
\]
because $da/dx=1/A$, $db/dx=-1/B$, and
\[
e_{r-1}(a^{A-1},b^B)-e_{r-1}(a^A,b^{B-1})
=
(b-a)e_{r-2}(a^{A-1},b^{B-1}).
\]
Here $e_j(a^M,b^N)$ denotes the $j$-th elementary symmetric sum of the multiset with $M$ copies of $a$ and $N$ copies of $b$.

The two monochromatic correction terms contribute
\[
-\frac{r}{q+1}\binom{q+1}{r}a^{r-1}
+
\frac{r}{q}\binom qr b^{r-1}
=
-\binom q{r-1}a^{r-1}
+
\binom{q-1}{r-1}b^{r-1}.
\]
Therefore
\[
F'(x)
=
(b-a)S_{r-2}(a,b)
-
\binom q{r-1}a^{r-1}
+
\binom{q-1}{r-1}b^{r-1}.
\]

If $x\ge\xi$, then $a\ge b$.
Hence the first term in \eqref{eq:derivative} is nonpositive, and
\[
        \binom{q-1}{r-1}b^{r-1}
        \le
        \binom q{r-1}a^{r-1}.
\]
Thus $F'(x)\le0$ for $x\ge\xi$.

It remains to prove $F'(x)\ge0$ for $0<x\le\tau$.
The endpoint $x=0$ follows by continuity.
In this range,
\[
        R=\frac ba\ge\frac{q+1}{q}=:\rho.
\]
Dividing \eqref{eq:derivative} by $a^{r-1}$, we get
\begin{equation}
\frac{F'(x)}{a^{r-1}}
=
(R-1)
\sum_{j=0}^{r-2}
\binom{A-1}{j}\binom{B-1}{r-2-j}R^{r-2-j}
-
\binom q{r-1}
+
\binom{q-1}{r-1}R^{r-1}.
\label{eq:normalized}
\end{equation}
The right-hand side is increasing in $R\ge1$.
Hence it suffices to prove its nonnegativity at $R=\rho$.

Since $A-1\ge q$ and $B-1\ge q-1$, the sum in \eqref{eq:normalized} is at least
\[
        \sum_{j=0}^{r-2}
        \binom qj\binom{q-1}{r-2-j}\rho^{r-2-j}.
\]
Keeping only the two terms $j=r-2$ and $j=r-3$, it is enough to prove
\begin{equation}
        \frac1q\binom q{r-2}
        +
        \frac{q-1}{q}\rho\binom q{r-3}
        +
        \binom{q-1}{r-1}\rho^{r-1}
        \ge
        \binom q{r-1}.
\label{eq:rho-bound}
\end{equation}
Let $m=r-1$.
By Bernoulli's inequality,
\[
        \rho^m=\left(1+\frac1q\right)^m\ge 1+\frac mq.
\]
Thus the left-hand side of \eqref{eq:rho-bound} is at least
\[
        \frac1q\binom q{m-1}
        +
        \left(1-\frac1{q^2}\right)\binom q{m-2}
        +
        \binom{q-1}{m}
        +
        \frac mq\binom{q-1}{m}.
\]
Since
\[
        \binom qm=\binom{q-1}{m}+\binom{q-1}{m-1},
\]
it remains to verify
\begin{equation}
        \frac1q\binom q{m-1}
        +
        \left(1-\frac1{q^2}\right)\binom q{m-2}
        +
        \frac mq\binom{q-1}{m}
        \ge
        \binom{q-1}{m-1}.
\label{eq:reduced}
\end{equation}
Put $d=q-m+1\ge1$.
After dividing by $\binom{q-1}{m-1}$, inequality \eqref{eq:reduced} becomes
\[
        \frac1d
        +
        \left(1-\frac1{q^2}\right)\frac{q(m-1)}{d(d+1)}
        \ge
        \frac mq .
\]
Multiplying by $qd(d+1)$, this is equivalent to
\[
        q(d+1)+(q^2-1)(m-1)-md(d+1)\ge0.
\]
Since $q=m+d-1$, the left-hand side is
\[
        (m-1)^2(2d+m-1)\ge0.
\]
Thus $F'(x)\ge0$ for $0<x\le\tau$.
\Cref{claim:localize-maximum} follows.

\medskip

We now work on the interval $[\tau,\xi]$.
Set
\[
        R=\frac ba=1+\frac{s}{q},
        \qquad 0\le s\le1.
\]
Then
\[
        a=\frac1{A+BR},
        \qquad
        b=\frac{R}{A+BR}.
\]
Define
\[
        N(R)
        =
        \sum_{j=0}^{r}
        \binom Aj\binom B{r-j}R^{r-j}
        -
        t\binom{q+1}{r}
        -
        (k-t)\binom q r R^r.
\]
Then
\[
        F(x)=\frac{N(R)}{(A+BR)^r}.
\]

Put
\[
        \nu=\frac nk=q+\frac tk,
        \qquad
        C=\binom nr-k\binom{\nu}{r}.
\]
The desired inequality $F(x)\le Cn^{-r}$ is equivalent to
\[
        P(s)
        :=
        C\bigl(A+B(1+s/q)\bigr)^r
        -
        n^r N(1+s/q)
        \ge0.
\]
Write the Bernstein expansion
\[
        P(s)=
        \sum_{m=0}^{r}\Gamma_m\binom rm s^m(1-s)^{r-m}.
\]

\begin{claim}\label{claim:Bernstein-coeff}
With the notation of Proposition~\ref{prop:pone-evaluation-Qkr}, the Bernstein coefficients of $P(s)$ are given by
\[
\Gamma_m
=
C n^{r-m}(k(q+1))^m
-
n^r\left[
\frac{\binom{n-m}{r-m}}{\binom rm}S_m
-
t\binom{q+1}{r}
-
(k-t)\binom qr\rho^m
\right].
\]
\end{claim}

We compute the degree-$r$ Bernstein coefficients of each term in
\[
        P(s)=C\bigl(A+B(1+s/q)\bigr)^r
        -n^r N(1+s/q)
\]
separately.
Recall that the degree-$r$ Bernstein basis is
\[
        \binom rm s^m(1-s)^{r-m},
        \qquad 0\le m\le r.
\]

First,
\[
        A+B(1+s/q)
        =
        A+B+\frac{B}{q}s
        =
        n+(k-t)s.
\]
Since
\[
        k(q+1)-n=k(q+1)-(kq+t)=k-t,
\]
we may rewrite this as
\[
        A+B(1+s/q)=n(1-s)+k(q+1)s.
\]
Therefore
\[
\bigl(A+B(1+s/q)\bigr)^r
=
\sum_{m=0}^r
\binom rm
n^{r-m}\bigl(k(q+1)\bigr)^m
s^m(1-s)^{r-m}.
\]
Thus the $m$-th Bernstein coefficient of the first term $C(A+B(1+s/q))^r$ is
\[
        C n^{r-m}\bigl(k(q+1)\bigr)^m .
\]

It remains to compute the Bernstein coefficients of $N(1+s/q)$.
Put
\[
        \rho=\frac{q+1}{q},
        \qquad
        R=1+\frac{s}{q}=(1-s)+\rho s.
\]
Let $\Omega$ be a set of size $n=A+B$, split into two parts
\[
        \Omega_A,\qquad \Omega_B,
        \qquad |\Omega_A|=A,\quad |\Omega_B|=B.
\]
Define
\[
        \theta_u=
        \begin{cases}
        1, & u\in \Omega_A,\\
        \rho, & u\in \Omega_B.
        \end{cases}
\]
Then
\[
        S_m=e_m(\theta_u:u\in\Omega)
        =
        \sum_{i=0}^m \binom Bi\binom A{m-i}\rho^i.
\]

The first part of $N(R)$ is
\[
        \sum_{j=0}^r \binom Aj\binom B{r-j}R^{r-j}.
\]
This can be written as
\[
        \sum_{E\in\binom{\Omega}{r}}
        \prod_{u\in E}\bigl((1-s)+\theta_u s\bigr).
\]
Indeed, a set $E$ with $j$ vertices in $\Omega_A$ and $r-j$ vertices in $\Omega_B$ contributes $R^{r-j}$.

For each fixed $E\in\binom{\Omega}{r}$,
\[
\prod_{u\in E}\bigl((1-s)+\theta_u s\bigr)
=
\sum_{m=0}^r
\left(
\sum_{\substack{T\subseteq E\\ |T|=m}}
\prod_{u\in T}\theta_u
\right)
s^m(1-s)^{r-m}.
\]
Hence its $m$-th Bernstein coefficient is
\[
        \frac{1}{\binom rm}
        \sum_{\substack{T\subseteq E\\ |T|=m}}
        \prod_{u\in T}\theta_u.
\]
Summing over all $E\in\binom{\Omega}{r}$, the $m$-th Bernstein coefficient of the first part of $N(1+s/q)$ is
\[
\frac{1}{\binom rm}
\sum_{E\in\binom{\Omega}{r}}
\sum_{\substack{T\subseteq E\\ |T|=m}}
\prod_{u\in T}\theta_u.
\]
Interchanging the order of summation, each fixed $m$-set $T\subseteq\Omega$ is contained in exactly $\binom{n-m}{r-m}$ sets $E\in\binom{\Omega}{r}$.
Therefore this coefficient equals
\[
        \frac{\binom{n-m}{r-m}}{\binom rm}
        \sum_{T\in\binom{\Omega}{m}}\prod_{u\in T}\theta_u
        =
        \frac{\binom{n-m}{r-m}}{\binom rm}S_m.
\]

Now consider the two monochromatic correction terms in $N(R)$.
The constant term
\[
        t\binom{q+1}{r}
\]
has every degree-$r$ Bernstein coefficient equal to
\[
        t\binom{q+1}{r},
\]
because
\[
        1=\sum_{m=0}^r \binom rm s^m(1-s)^{r-m}.
\]
Also,
\[
        R^r=((1-s)+\rho s)^r
        =
        \sum_{m=0}^r \binom rm \rho^m s^m(1-s)^{r-m},
\]
so the $m$-th Bernstein coefficient of
\[
        (k-t)\binom qr R^r
\]
is
\[
        (k-t)\binom qr\rho^m.
\]
Consequently, the $m$-th Bernstein coefficient of $N(1+s/q)$ is
\[
        \frac{\binom{n-m}{r-m}}{\binom rm}S_m
        -
        t\binom{q+1}{r}
        -
        (k-t)\binom qr\rho^m.
\]

Combining this with the first-term computation and using
\[
        P(s)=C\bigl(A+B(1+s/q)\bigr)^r-n^rN(1+s/q),
\]
we obtain
\[
\Gamma_m
=
C n^{r-m}(k(q+1))^m
-
n^r\left[
\frac{\binom{n-m}{r-m}}{\binom rm}S_m
-
t\binom{q+1}{r}
-
(k-t)\binom qr\rho^m
\right],
\]
as claimed.

\medskip

\begin{claim}\label{claim:positive-lower-bernstein-coefficients}
For every $0\le m\le r-1$, one has $\Gamma_m>0$.
\end{claim}

\smallskip

Let
\[
        \rho=\frac{q+1}{q},
        \qquad
        L=\frac{k(q+1)}{n}
        =
        \frac{q+1}{\nu}.
\]
A direct Bernstein-coefficient computation gives
\begin{equation}
\begin{aligned}
        \Gamma_m
        &=
        C\,n^{r-m}\bigl(k(q+1)\bigr)^m  \\
        &\quad
        -
        n^r\left[
        \frac{\binom{n-m}{r-m}}{\binom rm}S_m
        -
        t\binom{q+1}{r}
        -
        (k-t)\binom q r\rho^m
        \right],
\end{aligned}
\label{eq:gamma}
\end{equation}
where
\[
        S_m
        =
        \sum_{i=0}^{m}\binom Bi\binom A{m-i}\rho^i .
\]
The quantity $S_m$ is the $m$-th elementary symmetric sum of the vector with $A$ entries equal to $1$ and $B$ entries equal to $\rho$.
Its average is
\[
        \frac{A+B\rho}{n}=L.
\]
By Maclaurin's inequality,
\[
        S_m\le \binom nm L^m.
\]
Using
\[
        \binom nm\binom{n-m}{r-m}
        =
        \binom nr\binom rm,
\]
we get from \eqref{eq:gamma}
\begin{equation}
\frac{\Gamma_m}{n^r}
\ge
(q+1)^m
\left[
        t\frac{\binom{q+1}{r}}{(q+1)^m}
        +(k-t)\frac{\binom q r}{q^m}
        -k\frac{\binom{\nu}{r}}{\nu^m}
\right].
\label{eq:gamma-lower}
\end{equation}

It remains to prove convexity of
\[
        g_m(y)=\frac{\binom yr}{y^m}
        \qquad (0\le m\le r-1)
\]
on $[r-1,\infty)$.
Up to the positive factor $1/r!$,
\[
        g_m(y)=y^{1-m}(y-1)(y-2)\cdots(y-r+1).
\]
For $y>r-1$, set
\[
        a=\frac1y,\qquad
        a_i=\frac1{y-i}\quad(1\le i\le r-1),
\]
and
\[
        S=\sum_{i=1}^{r-1}a_i,\qquad
        S_2=\sum_{i=1}^{r-1}a_i^2,\qquad
        h=m-1.
\]
Then
\[
        \frac{g_m''(y)}{g_m(y)}
        =
        S^2-S_2-2haS+h(h+1)a^2.
\]
Put $y_i=a_i/a=y/(y-i)>1$ and $s_0=r-1$.
Dividing by $a^2$, we obtain
\begin{equation}
        \frac{g_m''(y)}{g_m(y)a^2}
        =
        2\sum_{i<j}y_i y_j
        -
        2h\sum_i y_i
        +
        h(h+1).
\label{eq:convexity}
\end{equation}
If $h\le0$, the right-hand side is positive.
If $0\le h\le s_0-1$, then
\[
        \sum_{i<j}y_i y_j
        \ge
        (s_0-1)\sum_i y_i-\binom{s_0}{2},
\]
because
\[
        \sum_{i<j}(y_i-1)(y_j-1)>0.
\]
Substitution in \eqref{eq:convexity}, together with $\sum_i y_i\ge s_0$, gives
\[
        \frac{g_m''(y)}{g_m(y)a^2}
        \ge
        2s_0(s_0-1-h)-s_0(s_0-1)+h(h+1).
\]
Writing $d=s_0-1-h\ge0$, the right-hand side is $d(d+1)\ge0$.
In the borderline case $d=0$, the earlier strict inequality $\sum_{i<j}(y_i-1)(y_j-1)>0$ gives strict positivity.
Hence $g_m''(y)>0$ on $(r-1,\infty)$.

Since
\[
        \nu=(1-t/k)q+(t/k)(q+1),
\]
Jensen's inequality gives
\[
        t g_m(q+1)+(k-t)g_m(q)>k g_m(\nu),
\]
because $1\le t\le k-1$.
This is exactly the strict positivity of the bracket in \eqref{eq:gamma-lower}.
Therefore $\Gamma_m>0$ for $0\le m\le r-1$.

\medskip

It remains to prove $\Gamma_r>0$.
Since at $s=1$ only the highest Bernstein coefficient remains, this is equivalent to proving $P(1)>0$.
The point $s=1$ corresponds to $x=\tau=t/k$, namely to equal class masses.
We call this the endpoint inequality.

At the endpoint, every colour class has total mass $1/k$.
After multiplying by $k^r$, the desired endpoint inequality is
\begin{equation}
[z^r]\left(1+\frac z{q+1}\right)^{t(q+1)}
       \left(1+\frac zq\right)^{(k-t)q}
-
t\psi_r(q+1)-(k-t)\psi_r(q)
\le
\binom nr\nu^{-r}-k\psi_r(\nu),
\label{eq:endpoint}
\end{equation}
where
\[
        \psi_r(y)=\binom yr y^{-r}.
\]
Equivalently, define
\[
        D_{\mathrm{all}}
        =
        \binom nr\nu^{-r}
        -
        [z^r]\left(1+\frac z{q+1}\right)^{t(q+1)}
              \left(1+\frac zq\right)^{(k-t)q},
\]
and
\[
        D_{\mathrm{mono}}
        =
        k\psi_r(\nu)-t\psi_r(q+1)-(k-t)\psi_r(q).
\]
It is enough to prove
\begin{equation}
        D_{\mathrm{all}}>D_{\mathrm{mono}}.
\label{eq:defect}
\end{equation}

Let
\[
        \alpha=\frac tk,
        \qquad
        \nu=q+\alpha.
\]

\medskip

\begin{claim}\label{claim:maclaurin-stability}
If $w_1,\ldots,w_N\ge m>0$ and $\sum_i w_i=Nc$, then
\begin{equation}
        \binom Nr c^r-e_r(w_1,\ldots,w_N)
        \ge
        \frac12\binom{N-2}{r-2}m^{r-2}
        \sum_{i=1}^{N}(w_i-c)^2.
\label{eq:stability}
\end{equation}
\end{claim}

Indeed, replacing two coordinates $x,y$ by their average increases $e_r$ by exactly
\[
        \frac{(x-y)^2}{4}e_{r-2}(\text{remaining coordinates}).
\]
Since all coordinates remain at least $m$, this increase is at least
\[
        \frac{(x-y)^2}{4}\binom{N-2}{r-2}m^{r-2}.
\]
Now put
\[
        V(w)=\sum_{i=1}^N (w_i-c)^2 .
\]
For one averaging step, suppose that two coordinates \(x,y\) are replaced by
\((x+y)/2,(x+y)/2\). The total sum is preserved, all coordinates remain at least
\(m\), and
\[
        V_{\rm before}-V_{\rm after}=\frac{(x-y)^2}{2}.
\]
On the other hand, as observed above,
\[
        e_r{}_{\rm after}-e_r{}_{\rm before}
        =
        \frac{(x-y)^2}{4}e_{r-2}(\text{remaining coordinates})
        \ge
        \frac{(x-y)^2}{4}\binom{N-2}{r-2}m^{r-2}.
\]
Consequently every averaging step satisfies
\[
        e_r{}_{\rm after}-e_r{}_{\rm before}
        \ge
        \frac12\binom{N-2}{r-2}m^{r-2}
        \bigl(V_{\rm before}-V_{\rm after}\bigr).
\]

We now apply these averaging steps repeatedly, in a fixed cyclic order over all
unordered pairs of coordinates. We claim that the resulting sequence converges to the
constant vector \((c,\ldots,c)\). Indeed, \(V\) is nonincreasing and bounded below, hence
it has a limit. Consider the sequence at the beginnings of complete cycles. If a
subsequential limit \(z\) were not constant, then one full cycle of pair-averaging would
strictly decrease \(V\). By continuity of the full-cycle map, the same would hold with a
uniform positive decrease for all sufficiently close vectors, contradicting the existence of
the limit of \(V\). Thus every subsequential limit is constant. Since the total sum is
always \(Nc\), the only possible constant limit is \((c,\ldots,c)\). Hence the whole
sequence converges to \((c,\ldots,c)\).

Summing the preceding one-step inequality over the first \(T\) averaging steps gives
\[
        e_r(w^{(T)})-e_r(w^{(0)})
        \ge
        \frac12\binom{N-2}{r-2}m^{r-2}
        \bigl(V(w^{(0)})-V(w^{(T)})\bigr).
\]
Letting \(T\to\infty\), and using the convergence \(w^{(T)}\to(c,\ldots,c)\), we obtain
\[
        \binom Nr c^r-e_r(w^{(0)})
        \ge
        \frac12\binom{N-2}{r-2}m^{r-2}
        \sum_{i=1}^N(w_i^{(0)}-c)^2.
\]
This proves the claim.

In our application, $N=n$, and the vector $w$ consists of
\[
        t(q+1) \text{ entries equal to } \frac1{q+1},
        \qquad
        (k-t)q \text{ entries equal to } \frac1q.
\]
Its average is $1/\nu$, and its minimum is $1/(q+1)$.
A direct calculation gives
\[
        \sum_i\left(w_i-\frac1\nu\right)^2
        =
        \frac{k\alpha(1-\alpha)}{q(q+1)\nu}.
\]
Thus \eqref{eq:stability} yields
\begin{equation}
        D_{\mathrm{all}}
        \ge
        \frac{k\alpha(1-\alpha)}2
        \frac{\binom{n-2}{r-2}}{q(q+1)^{r-1}\nu}.
\label{eq:all-defect}
\end{equation}

\medskip

\begin{claim}\label{claim:differential-estimate}
For every $r\ge3$ and $y\ge r-1$,
\begin{equation}
        -\frac{d^2}{dy^2}\left(\binom yr y^{-r}\right)
        \le
        \frac{y-r+1}{(r-2)!y^4}.
\label{eq:differential}
\end{equation}
\end{claim}

Let $X=y-r+1$, $Y=y=X+r-1$, and
\[
        P_r(X)=X(X+1)\cdots(X+r-1).
\]
Then
\[
        \binom yr y^{-r}=\frac{P_r(X)}{r!Y^r}.
\]
A direct differentiation shows that \eqref{eq:differential} is equivalent to
\[
        E_r(X)\ge0,
\]
where
\begin{equation}
        E_r(X)
        =
        Y^2P_r''-2rYP_r'+r(r+1)P_r+r(r-1)XY^{r-2}.
\label{eq:e-polynomial}
\end{equation}
We prove that $E_r$ has nonnegative coefficients.

For $s\ge3$, define
\[
        P_s(X)=X(X+1)\cdots(X+s-1),
\]
and define $E_s$ by the analogue of \eqref{eq:e-polynomial}.
Directly,
\[
        E_3(X)=0.
\]
Using $P_{s+1}=(X+s)P_s$, a direct calculation gives
\begin{equation}
        E_{s+1}(X)=(X+s)(E_s(X)+\Delta_s(X)),
\label{eq:recurrence}
\end{equation}
where
\begin{equation}
\begin{aligned}
        \Delta_s(X)
        &=
        (2X+2s-1)P_s''-2sP_s'  \\
        &\quad
        +sX\left((s+1)(X+s)^{s-2}
        -(s-1)(X+s-1)^{s-2}\right).
\end{aligned}
\label{eq:delta}
\end{equation}
It remains to prove that $\Delta_s$ has nonnegative coefficients.

Write
\[
        P_s(X)=\sum_{j=1}^{s}c_jX^j.
\]
Then
\[
        c_j=e_{s-j}(1,2,\ldots,s-1).
\]
Let
\[
        A_s(X)=(2X+2s-1)P_s''-2sP_s'.
\]
For $0\le m\le s-2$, the coefficient of $X^m$ in $A_s$ is
\begin{equation}
        (m+1)\left((2s-1)(m+2)c_{m+2}
        -2(s-m)c_{m+1}\right).
\label{eq:coefficient}
\end{equation}
Set $\ell=s-m-1$.
Then
\[
        c_{m+1}=e_\ell(1,\ldots,s-1),
        \qquad
        c_{m+2}=e_{\ell-1}(1,\ldots,s-1).
\]

By Newton's inequalities, the normalized elementary symmetric means
\[
        \frac{e_j(1,\ldots,s-1)}{\binom{s-1}{j}}
\]
form a log-concave sequence.
Hence
\[
        \frac{e_\ell(1,\ldots,s-1)}{e_{\ell-1}(1,\ldots,s-1)}
        \le
        \frac{\binom{s-1}{\ell}}{\binom{s-1}{\ell-1}}
        \cdot
        \frac{e_1(1,\ldots,s-1)}{s-1}.
\]
Since $e_1(1,\ldots,s-1)=s(s-1)/2$, this gives
\begin{equation}
        \frac{c_{m+1}}{c_{m+2}}
        \le
        \frac{s(m+1)}{2(s-m-1)}.
\label{eq:ratio}
\end{equation}
Also,
\begin{equation}
        \frac{s(m+1)}{2(s-m-1)}
        \le
        \frac{(2s-1)(m+2)}{2(s-m)}.
\label{eq:bound}
\end{equation}
Indeed, with $\ell=s-m-1$, inequality \eqref{eq:bound} is equivalent to
\[
        \ell(2s-1)(s-\ell+1)-s(s-\ell)(\ell+1)\ge0.
\]
Writing $\ell=d+1$, this becomes
\[
        s+d(s^2+1)-d^2(s-1)\ge0,
\]
a concave quadratic in $d\in[0,s-2]$, nonnegative at both endpoints.

Combining \eqref{eq:ratio} and \eqref{eq:bound}, all coefficients in \eqref{eq:coefficient} are nonnegative.
The coefficient of $X^{s-1}$ in $A_s$ is $-2s$.

Now set
\[
        B_s(X)
        =
        sX\left((s+1)(X+s)^{s-2}
        -(s-1)(X+s-1)^{s-2}\right).
\]
The polynomial $B_s$ has nonnegative coefficients, because for each $j$,
\[
        (s+1)s^{s-2-j}-(s-1)(s-1)^{s-2-j}\ge0.
\]
Moreover, the coefficient of $X^{s-1}$ in $B_s$ is $2s$.
Hence
\[
        \Delta_s=A_s+B_s
\]
has nonnegative coefficients.
By induction from \eqref{eq:recurrence}, $E_r$ has nonnegative coefficients.
This proves \eqref{eq:differential}.

\medskip

By the standard second-order interpolation estimate, that is, if $f\in C^2([q,q+1])$ and $0\le\alpha\le1$, then
\[
f(q+\alpha)-\bigl((1-\alpha)f(q)+\alpha f(q+1)\bigr)
\le
\frac{\alpha(1-\alpha)}2
\max_{y\in[q,q+1]}(-f''(y))_+ .
\]
applied to $\psi_r(y)=\binom yr y^{-r}$,
\[
        D_{\mathrm{mono}}
        \le
        \frac{k\alpha(1-\alpha)}2
        \max_{y\in[q,q+1]}\bigl(-\psi_r''(y)\bigr)_+.
\]

Using \eqref{eq:differential},
\begin{equation}
        D_{\mathrm{mono}}
        \le
        \frac{k\alpha(1-\alpha)}2
        \max_{y\in[q,q+1]}
        \frac{y-r+1}{(r-2)!y^4}.
\label{eq:mono-defect}
\end{equation}

\medskip

\begin{claim}\label{claim:scalar-endpoint-estimate}
For every $r\ge3$ and $q\ge r-1$,
\begin{equation}
        \max_{y\in[q,q+1]}
        \frac{y-r+1}{(r-2)!y^4}
        <
        \frac{\binom{2q-2}{r-2}}{q^2(q+1)^{r-1}}.
\label{eq:scalar}
\end{equation}
\end{claim}

Let $a_0=r-1$.
Inequality \eqref{eq:scalar} is equivalent to
\begin{equation}
        \max_{y\in[q,q+1]}\frac{y-a_0}{y^4}
        <
        \frac{(2q-2)_{a_0-1}}{q^2(q+1)^{a_0}},
\label{eq:equivalent}
\end{equation}
where $(x)_s=x(x-1)\cdots(x-s+1)$.
Set
\[
        h(y)=\frac{y-a_0}{y^4}.
\]
Then
\[
        h'(y)=\frac{4a_0-3y}{y^5},
\]
so $h$ is maximized at $y=4a_0/3$, with
\begin{equation}
        h(4a_0/3)=\frac{27}{256a_0^3}.
\label{eq:maximum}
\end{equation}
Since $q\ge a_0$, for $0\le i\le a_0-2$,
\[
        2q-2-i\ge2q-a_0\ge q.
\]
Thus
\[
        (2q-2)_{a_0-1}\ge q^{a_0-1},
\]
and hence
\begin{equation}
        \frac{(2q-2)_{a_0-1}}{q^2(q+1)^{a_0}}
        \ge
        \frac1{q^3(1+1/q)^{a_0}}.
\label{eq:falling}
\end{equation}

If $q\le 4a_0/3$, then
\[
        (1+1/q)^{a_0}\le (1+1/a_0)^{a_0}<4,
\]
and so by \eqref{eq:falling},
\[
        \frac{(2q-2)_{a_0-1}}{q^2(q+1)^{a_0}}
        >
        \frac1{4q^3}
        \ge
        \frac{27}{256a_0^3}.
\]
Together with \eqref{eq:maximum}, this proves \eqref{eq:equivalent} in this case.

If $q\ge4a_0/3$, then $h$ is decreasing on $[q,q+1]$, and its maximum is
\[
        h(q)=\frac{q-a_0}{q^4}.
\]
By \eqref{eq:falling}, it is enough to prove
\[
        \frac1{q^3(1+1/q)^{a_0}}\ge \frac{q-a_0}{q^4}.
\]
Equivalently,
\[
        (1-a_0/q)(1+1/q)^{a_0}\le1.
\]
Let $\theta=a_0/q\in(0,1)$.
Taking logarithms,
\[
\begin{aligned}
        \log\left((1-\theta)(1+1/q)^{a_0}\right)
        &=
        \log(1-\theta)+a_0\log(1+1/q)\\
        &<
        \log(1-\theta)+\theta
        \le0,
\end{aligned}
\]
because $q\log(1+1/q)<1$.
This proves \eqref{eq:equivalent}, and hence \eqref{eq:scalar}.

Finally, since $1\le t\le k-1$, we have $\nu=q+t/k<q+1$ and
\[
        n=kq+t\ge 2q+1 .
\]
We claim that
\begin{equation}
        \frac{\binom{2q-2}{r-2}}{q^2(q+1)^{r-1}}
        <
        \frac{\binom{n-2}{r-2}}{q(q+1)^{r-1}\nu}.
\label{eq:comparison}
\end{equation}
Equivalently, it is enough to prove
\begin{equation}
        \frac q\nu \binom{n-2}{r-2}
        >
        \binom{2q-2}{r-2}.
\label{eq:sufficient}
\end{equation}
We prove this in two cases.

First assume $r\ge 4$.
Since $n\ge 2q+1$, we have
\[
        \binom{n-2}{r-2}
        \ge
        \binom{2q-1}{r-2}.
\]
Since $\nu<q+1$, it follows that
\[
        \frac q\nu \binom{n-2}{r-2}
        >
        \frac q{q+1}\binom{2q-1}{r-2}.
\]
Thus it remains to show that
\[
        \frac q{q+1}\binom{2q-1}{r-2}
        \ge
        \binom{2q-2}{r-2}.
\]
Using
\[
        \binom{2q-1}{r-2}
        =
        \frac{2q-1}{2q-r+1}
        \binom{2q-2}{r-2},
\]
this is equivalent to
\[
        \frac{q(2q-1)}{(q+1)(2q-r+1)}\ge 1.
\]
The numerator minus the denominator is
\[
        q(2q-1)-(q+1)(2q-r+1)
        =
        (r-4)q+(r-1)\ge 0,
\]
because $r\ge 4$.
Hence \eqref{eq:sufficient}, and therefore \eqref{eq:comparison}, holds when $r\ge4$.

It remains to consider $r=3$.
In this case \eqref{eq:sufficient} becomes
\[
        \frac q\nu (n-2)>2q-2.
\]
Since $\nu=n/k$, this is equivalent to
\[
        qk(n-2)>(2q-2)n.
\]
Substituting $n=kq+t$, the difference between the left-hand side and the right-hand side is
\[
\begin{aligned}
        qk(n-2)-(2q-2)n
        &= qk(kq+t-2)-(2q-2)(kq+t)  \\
        &= (k-2)q(kq+t)+2t  \\
        &>0,
\end{aligned}
\]
because $k\ge 2$ and $t\ge 1$.
Therefore \eqref{eq:sufficient}, and hence \eqref{eq:comparison}, also holds for $r=3$.
This proves \eqref{eq:comparison}.

Combining \eqref{eq:all-defect}, \eqref{eq:mono-defect}, \eqref{eq:scalar}, and \eqref{eq:comparison}, we obtain
\[
    D_{\mathrm{all}}>D_{\mathrm{mono}}.
\]
Thus $P(1)>0$, so
\[
        \Gamma_r>0.
\]

All Bernstein coefficients $\Gamma_0,\ldots,\Gamma_r$ are therefore positive.
Hence
\[
        P(s)>0
        \qquad(0\le s\le1).
\]
Consequently,
\[
        F(x)
        <
        \left(\binom nr-k\binom{n/k}{r}\right)n^{-r}
        \qquad(\tau\le x\le\xi).
\]
By \cref{claim:localize-maximum}, the maximum of $F$ occurs in this interval.
Thus
\[
        \lambda^{(1)}(Q_k^r(n))
        <
        r!\left(\binom nr-k\binom{n/k}{r}\right)n^{-r}
\]
when $1\le t\le k-1$.
This proves the strict case.
The equality case $t=0$ was proved earlier.
Therefore \cref{prop:pone-evaluation-Qkr} is proved.
\end{proof}

\section{Concluding remarks}
\label{sec:concluding}

The results of this paper settle the $p$-spectral extremal problem for the complete $k$-chromatic pattern.
We close with two directions in which the methods and questions of the paper may be extended.

\subsection*{Pattern and mixed-pattern constructions}

It is natural to ask whether analogous spectral extremal results can be proved for more general recursive constructions.
We use the language of patterns and mixed pattern constructions introduced in \cite{Pik14,LiuPikhurko2025}.

Recall that an $r$-graph pattern is a triple $P=(m,E,R)$, where $E$ is a collection of $r$-multisets on $[m]$ and $R\subseteq[m]$ is the set of recursive indices.
A $P$-construction is obtained recursively: one first takes a blow-up of the profile family $E$, and then inserts further $P$-constructions inside the parts indexed by $R$.
More generally, if $\mathcal P=\{P_i:i\in I\}$ is a finite family of patterns, then a $\mathcal P$-mixing construction is obtained by allowing the choice of the pattern $P_i$ to vary at each recursive step.

The present paper treats the spectral optimization of one special pattern, namely the complete $k$-chromatic pattern.
For a general pattern, or a finite family of mixed patterns, the corresponding $p$-spectral problem seems largely open.

\begin{problem}
\label{prob:pattern-spectral-radius}
Let $r\ge3$, $p\ge1$, and let $\mathcal P$ be a finite family of $r$-graph patterns.
Let $\Sigma_{\mathcal P}$ be the family of all $\mathcal P$-mixing constructions.
Determine
\[
        \Lambda_{\mathcal P}^{(p)}(n)
        =
        \max\left\{
        \lambda^{(p)}(G):
        v(G)=n,\; G\in\Sigma_{\mathcal P}
        \right\}.
\]
\end{problem}

\subsection*{Other notions of independence and colorability}

Another natural direction is to vary the notion of independence used in the definition of colorability.
The present paper uses the standard weak notion: a set is independent if it contains no edge, and a $k$-coloring is a partition of the vertex set into $k$ independent sets.
Equivalently, each edge is allowed to contain at most $r-1$ vertices from any one color class.
At the other extreme, in a strong coloring every edge has all its vertices in distinct color classes; equivalently, each color class meets every edge in at most one vertex.

These notions can be interpolated as follows.
For $1\le \ell\le r-1$, call a set $U\subseteq V(G)$ $\ell$-independent if
\[
        |e\cap U|\le \ell
        \qquad\text{for every }e\in G.
\]
Let $\chi_\ell(G)$ be the least $k$ such that $V(G)$ can be partitioned into $k$ $\ell$-independent sets.
Thus $\chi_{r-1}$ is the usual weak chromatic number considered in this paper, while $\chi_1$ is the strong chromatic number.

Given a partition $V=V_1\cup\cdots\cup V_k$, define the complete $(k,\ell)$-colorable $r$-graph on this partition to be the $r$-graph whose edges are precisely the $r$-sets $e$ satisfying
\[
        |e\cap V_i|\le \ell
        \qquad\text{for every }i\in[k].
\]
Let $Q_{k,\ell}^r(n)$ denote the member with color classes as equal as possible.
Then
\[
        Q_{k,r-1}^r(n)=Q_k^r(n),
\]
the graph studied in this paper, while $Q_{k,1}^r(n)$ is the balanced complete $k$-partite $r$-graph.

\begin{problem}
\label{prob:ell-colorability}
Let $r \ge 4$, $2\le \ell\le r-2$, $k\ge2$, and $p\ge1$.
Determine
\[
        \max\left\{
        \lambda^{(p)}(G):
        v(G)=n,\ \chi_\ell(G)\le k
        \right\}.
\]
In particular, decide when the unique extremal graph is the balanced complete $(k,\ell)$-colorable $r$-graph $Q_{k,\ell}^r(n)$.
\end{problem}


\section*{Declaration on the use of AI}

Proposition~\ref{prop:structural-hypotheses}, which implies \cite[Conjecture~7]{KNY2015} of Kang--Nikiforov--Yuan for $p\ge2$, was proved without the use of artificial intelligence tools.
The proof of the remaining range $1\le p<2$, using Proposition~\ref{prop:structural-hypotheses} as an input, and the derivation of \cite[Conjecture~8]{KNY2015} from \cite[Conjecture~7]{KNY2015} were obtained with the assistance of an artificial intelligence system.
These AI-assisted arguments were checked, edited, and incorporated by the authors.
\bibliographystyle{abbrv}
\bibliography{spectral}
\end{document}